\documentclass[pdftex]{amsart}
\usepackage{amsmath, amssymb, url}
\usepackage{type1cm}
\usepackage{tikz, tikz-cd}
\usetikzlibrary{matrix,arrows,shapes,calc}

\newtheorem{defi}{Definition}[section]
\newtheorem{thm}[defi]{Theorem}
\newtheorem{prop}[defi]{Proposition}
\newtheorem{lem}[defi]{Lemma}
\newtheorem{cor}[defi]{Corollary}
\newtheorem{eg}[defi]{Example}
\newtheorem{rem}[defi]{Remark}
\newtheorem{conj}[defi]{Conjecture}

\DeclareMathOperator{\Art}{(Art)}
\DeclareMathOperator{\Aut}{Aut}

\DeclareMathOperator{\CC}{\mathbb{C}}
\DeclareMathOperator{\codim}{codim}
\DeclareMathOperator{\coh}{coh}

\DeclareMathOperator{\D}{D^b}
\DeclareMathOperator{\Def}{Def}
\DeclareMathOperator{\depth}{depth}
\DeclareMathOperator{\Der}{Der}
\DeclareMathOperator{\EE}{\mathcal{E}}
\DeclareMathOperator{\End}{End}
\DeclareMathOperator{\EEnd}{\mathcal{E}\textit{nd}}
\DeclareMathOperator{\exc}{exc}
\DeclareMathOperator{\Ext}{Ext}

\DeclareMathOperator{\GL}{GL}

\DeclareMathOperator{\Gm}{\mathbb{G}_{m}}
\DeclareMathOperator{\Gr}{Gr}
\DeclareMathOperator{\HH}{HH}
\DeclareMathOperator{\Hilb}{Hilb}
\DeclareMathOperator{\Hom}{Hom}

\newcommand{\id}{\mathrm{id}}

\DeclareMathOperator{\Ker}{Ker}

\DeclareMathOperator{\modu}{mod}

\DeclareMathOperator{\Perf}{Perf}
\DeclareMathOperator{\PP}{\mathbb{P}}

\DeclareMathOperator{\Qcoh}{Qcoh}

\DeclareMathOperator{\RGamma}{R\Gamma}
\DeclareMathOperator{\RHom}{RHom}

\newcommand{\sh}{\mathcal{F}}
\DeclareMathOperator{\Sh}{Sh}
\DeclareMathOperator{\Sing}{Sing}

\DeclareMathOperator{\Spec}{Spec}

\DeclareMathOperator{\Supp}{Supp}

\DeclareMathOperator{\Sym}{Sym}

\newcommand{\stsh}{\mathcal{O}}

\title[deformation of tilting-type derived equivalences]{deformation of tilting-type derived equivalences for crepant resolutions}
\author{wahei hara}
\address{Department of Mathematics, School of Science and Engineering, Waseda University, 3-4-1 Ohkubo, Shinjuku, Tokyo 169-8555, Japan}
\email{waheyhey@ruri.waseda.jp}
\date{}

\begin{document}

\begin{abstract}
We say that an exact equivalence  between the derived categories of two algebraic varieties is \textit{tilting-type} if it is constructed by using tilting bundles.
The aim of this article is to understand the behavior of tilting-type equivalences for crepant resolutions under deformations.
As an application of the method that we establish in this article, we study the derived equivalence for stratified Mukai flops and stratified Atiyah flops
in terms of tilting bundles.
\end{abstract}

\maketitle

\tableofcontents

\section{Introduction}

\subsection{Background}
Let $X = \Spec R$ be a normal Gorenstein affine variety and assume that $X$ admits a crepant resolution.
Although $X$ may have some different crepant resolutions, the following ``uniqueness" is expected.

\begin{conj}[Bondal-Orlov]
Let $\phi : Y \to X$ and $\phi' : Y' \to X$ be two crepant resolutions of $X$.
Then $Y$ and $Y'$ are derived equivalent to each other, i.e. there exists an exact equivalence
\[ \Phi : \D(Y) \xrightarrow{\sim} \D(Y'). \]
\end{conj}

For given two crepant resolutions $Y$ and $Y'$ of $X$, there are various methods to construct a derived equivalence between $\D(Y)$ and $\D(Y')$
(eg. Fourier-Mukai transform \cite{Bri02, Cau12, Kaw02, Kaw05},
variation of geometric invariant theory (=VGIT) \cite{HL16},
or mutation of semi-orthogonal decomposition \cite{Ued16}).
In this article, we deal with the one using tilting bundles.
A vector bundle $E$ on a scheme $Z$ is called \textit{tilting bundle}
if $\Ext_Z^i(E,E) = 0$ for $i \neq 0$ and if $E$ is a generator of the category $\mathrm{D}^-(\Qcoh(Z))$.
It is well-known that, by using tilting bundles, we can construct equivalences of categories:

\begin{lem}[See also Proposition \ref{tilting equiv}] \label{lem intro}
If there are tilting bundles $T$ and $T'$ on $Y$ and $Y'$, respectively, with an $R$-algebra isomorphism
\[ \End_{Y}(T) \simeq \End_{Y'}(T'), \]
then we have equivalences of derived categories
\[ \D(Y) \simeq \D(\End_{Y}(T)) \simeq  \D(\End_{Y'}(T')) \simeq \D(Y'). \]
\end{lem}
In this article, we call a derived equivalence constructed in this way \textit{tilting-type}.
We say that a tilting-type equivalence $\D(Y) \xrightarrow{\sim} \D(Y')$ is \textit{strict} if the tilting bundles $T$ and $T'$ coincide with each other on the largest common
open subset $U$ of $X$, $Y$ and $Y'$.
We also say that a tilting-type equivalence is \textit{good} if tilting bundles contain trivial line bundles as direct summands.
Note that these terminologies are ad hoc.
Tilting-type equivalences that are good and strict constitute an important class of derived equivalences.
For example, it is known that if $X$ has only threefold terminal singularities, then there is a derived equivalence between $Y$ and $Y'$ that can be written as a composition of good and strict tilting-type equivalences.
Indeed, since two crepant resolutions are connected by iterating flops, this fact follows from the result of Van den Bergh \cite{VdB04a}.
In addition, tilting-type equivalences for crepant resolutions have a strong relationship with the theory of non-commutative crepant resolutions, which was first introduced by Van den Bergh \cite{VdB04b}.

On the other hand, taking \textit{a deformation of an algebraic variety} is one of standard methods to construct a new variety from the original one,
and is studied in many branches of algebraic geometry.
Taking deformations is also an important operation in Mirror Symmetry.
According to Homological Mirror Symmetry, the derived categories of algebraic varieties are quite significant objects in the study of Mirror Symmetry.
The aim of this article is to understand the behavior of (good or strict) tilting-type equivalences under deformations.

\subsection{Results}
Let $X_0$ be a normal Gorenstein affine variety, and let $\phi_0 : Y_0 \to X_0$ and $\phi'_0 : Y'_0 \to X_0$ be two crepant resolutions of $X_0$.
In this article, we deal with three types of deformations:
infinitesimal deformation, deformation over a complete local ring, and deformation with a $\Gm$-action.

\vspace{3mm}

\noindent
\textbf{$\bullet$ Infinitesimal deformation.}
First we study infinitesimal deformations of small resolutions.
Assume that
\[ \codim_{Y_0}(\exc(\phi_0)) \geq 3 ~ \text{and} ~  \codim_{Y'_0}(\exc(\phi'_0)) \geq 3. \]
Then we show that there are canonical isomorphisms of deformation functors
\[ \Def Y_0 \simeq  \Def X_0 \simeq \Def Y'_0 \]
(Proposition {prop 3-1}).
Let $A$ be a local Artinian algebra with residue field $\CC$,
and let $Y$ and $Y'$ be deformations of $Y_0$ and $Y'_0$ over $A$ which correspond to each other under the above isomorphisms.
Then we show the following.

\begin{thm}[= Theorem \ref{inf def main thm}] \label{intro thm inf}
Under the notations and assumptions above, any strict tilting-type equivalence between $\D(Y_0)$ and $\D(Y'_0)$ lifts to a strict tilting-type equivalence between $\D(Y)$ and $\D(Y')$.
\end{thm}

See Definition \ref{def lift tilting-type} for the precise meaning of the word \textit{lift}.

\vspace{3mm}
\noindent
\textbf{$\bullet$ Complete local or $\Gm$-equivariant deformation.}
We also study deformations over a complete local ring and deformations with $\Gm$-actions.
Let $X_0$, $Y_0$, $Y'_0$, $\phi_0$ and $\phi'_0$ as above.
(Note that we do NOT assume the condition for the codimension of the exceptional locus here.)
Consider a deformation of them
\[ \begin{tikzcd}
Y \arrow[r, "\phi"] \arrow[rd, "p"'] & X \arrow[d, "q"] & Y' \arrow[l, "\phi'"'] \arrow[ld, "p'"] \\
& (\Spec D, d) &
\end{tikzcd} \]
over a pointed affine scheme $(\Spec D, d)$, where $\phi$ and $\phi'$ are projective morphisms and $X = \Spec R$ is an affine scheme.
Assume that an inequality
\[ \codim_{X_0} \Sing(X_0) \geq 3 \]
holds and that one of the following conditions holds.
\begin{enumerate}
\item[(a)] $D$ is a complete local ring and $d \subset D$ is the maximal ideal.
\item[(b)] $X$, $Y$ and $Y'$ are $\Gm$-varieties, $\phi$ and $\phi'$ are $\Gm$-equivariant, and the action of $\Gm$ on $X$ is good.
For a unique $\Gm$-fixed point $x \in X$, we have $d = q(x)$.
\end{enumerate}
See Definition \ref{def good action} for the definition of good $\Gm$-actions.
Then we have a similar theorem as in the case of infinitesimal deformations.

\begin{thm}[= Theorem \ref{thm formal def}, \ref{main thm equiv}] \label{intro thm complete and equivariant}
Under the conditions above, any good and strict tilting-type equivalence between $\D(Y_0)$ and $\D(Y'_0)$
lifts to  a good tilting-type equivalence between $\D(Y)$ and $\D(Y')$.
\end{thm}

We note that we cannot generalize this theorem directly to the case when the codimension of the singular locus is two
(see Section \ref{sect Eg counter}).
As a direct corollary of the theorem above, we have the following.

\begin{cor}[= Corollary \ref{cor equiv defo fiber}]
Under the condition (b) above, assume that there exists a good and strict tilting-type equivalence between $\D(Y_0)$ and $\D(Y'_0)$.
Then, for any closed point $t \in \Spec D$, there is a good tilting-type equivalence between $\D(p^{-1}(t))$ and $\D(p'^{-1}(t))$.
\end{cor}

\noindent
\textbf{$\bullet$ Stratified Mukai flops and stratified Atiyah flops.}
As an application of the theorems above, we study derived equivalences for stratified Mukai flops and stratified Atiyah flops.
A \textit{stratified Mukai flop} on $\Gr(r,N)$ is a birational map $Y_0 \dashrightarrow Y'_0$ between the cotangent bundles $Y_0 := T^*\Gr(r,N)$ and $Y'_0 := T^*\Gr(N-r,N)$ of Grassmannian varieties, where $r$ is an integer with $2r \leq N -1$.
It is known that they have a natural one-parameter $\Gm$-equivariant deformation $Y \dashrightarrow Y'$ called \textit{stratified Atiyah flop} on $\Gr(r,N)$.
Note that a stratified Atiyah flop on $\Gr(r,N)$ is also defined in the case if $2r = N$
(see Section \ref{subsect; stratified flop} for more details).
Stratified Mukai flops and stratified Atiyah flops form a fundamental class of higher dimensional flops.

The method to construct an equivalence for stratified Mukai flops from an equivalence for stratified Atiyah flops is well-established (eg. \cite{Kaw02, Sze04}).
On the other hand, our theorem provides a method to construct a tilting-type equivalence for stratified Atiyah flops from a tilting-type equivalence for stratified Mukai flops.
More precisely:

\begin{cor}[= Theorem \ref{Mukai Atiyah lift}] \label{intro stratified}
Any good and strict tilting-type equivalence for the stratified Mukai flop on $\Gr(r,N)$ lifts to a good and strict tilting-type equivalence for the stratified Atiyah flop on $\Gr(r,N)$.
\end{cor}

We note that, if $2r \leq N - 2$, we can remove the assumption that the tilting-type equivalence is good.
Since we can construct a strict tilting-type equivalence for stratified Mukai flops using results of Kaledin \cite{Kal08},
we have the following corollary.
Although there are some previous works on the derived equivalence for stratified Atiyah flops (eg. \cite{Kaw05, Cau12}), the following corollary is new to the best of the author's knowledge.

\begin{cor}
If $2r \leq N -2$, then there exists a strict tilting-type equivalence for the stratified Atiyah flop on $\Gr(r,N)$.
\end{cor}

\subsection{Plan of the article}
In Section \ref{sect prelim} we give the definitions for some basic terminologies and provide their fundamental properties we use in the later sections.
In Section \ref{sect tilting-type equiv} we define (good or strict) tilting-type equivalence for schemes and investigate their properties.
In Section \ref{sect infinitesimal} we explain some infinitesimal deformation theory of crepant resolutions and prove Theorem \ref{intro thm inf}.
Section \ref{sect complete local} and \ref{sect equivariant} will be devoted to prove Theorem \ref{intro thm complete and equivariant}.
In Section \ref{sect example} we explain the geometry of stratified Mukai flops and stratified Atiyah flops, and prove Corollary \ref{intro stratified}.
In Appendix  \nolinebreak \ref{sect appendix} we discuss certain generalization of Kaledin's result on derived equivalences of symplectic resolutions.

\subsection{Notations. }
In this article, we always work over the complex number field $\mathbb{C}$. 
A \textit{scheme} always means a Noetherian $\CC$-scheme.
Moreover, we adopt the following notations.
\begin{enumerate}
\item[$\bullet$] $\Qcoh(X)$ : the category of quasi-coherent sheaves on a scheme $X$.
\item[$\bullet$] $\coh(X)$ : the category of coherent sheaves on a scheme $X$.
\item[$\bullet$] $\modu(A)$ : the category of finitely generated right modules over a ring $A$.
\item[$\bullet$] $\mathrm{D}^*(\mathcal{A})$, ($* = - ~ \text{or} ~ \mathrm{b}$) : the (bounded above or bounded) derived category of an abelian category $\mathcal{A}$.
\item[$\bullet$] $\mathrm{D}^*(X) := \mathrm{D}^*(\coh(X))$, ($* = - ~ \text{or} ~ \mathrm{b}$) : the (bounded above or bounded) derived category of $\coh(X)$.
\item[$\bullet$] $\mathrm{D}^*(A) := \mathrm{D}^*(\modu(A))$, ($* = - ~ \text{or} ~ \mathrm{b}$) : the (bounded above or bounded) derived category of $\modu(A)$.
\item[$\bullet$] $\Art$ : the category of local Artinian (commutative) $\CC$-algebras with residue field $\CC$.
\item[$\bullet$] $\Def X$ : the deformation functor (or local moduli functor) of $X$.
\item[$\bullet$] $\exc(\phi)$ : the exceptional locus of a birational morphism $\phi : Y \to X$.
\end{enumerate}
In addition, we refer to the bounded derived category $\D(X)$ of coherent sheaves on $X$ as \textit{the derived category of $X$}.

\vspace{3mm}

\noindent
\textbf{Acknowledgements.}
The author would like to express his gratitude to Professor Michel Van den Bergh for valuable comments and discussions,
and for reading the draft version of this article.
The author would also like to thank his supervisor Professor Yasunari Nagai for continuous encouragement,
and Yuki Hirano, Takeru Fukuoka and Naoki Koseki for reading the draft version of this article and giving useful comments.

This work was done during the author's stay at Hasselt University in Belgium.
The author would like to thank Hasselt University for the hospitality and excellent working condition.
This work is supported by Grant-in-Aid for JSPS Research Fellow 17J00857.

\section{Preliminaries} \label{sect prelim}

\subsection{Tilting bundles and Non-commutative crepant resolutions}

In the following, a \textit{vector bundle} on a scheme $X$ means a locally free sheaf of finite rank on $X$.

\begin{defi} \label{def tilting bundle} \rm
A vector bundle $T$ on a scheme $X$ is said to be \textit{partial tilting} if
\[ \Ext_X^p(T,T) = 0 \]
for all $p>0$.
A partial tilting bundle $T$ is called a \textit{tilting} bundle if $T$ is in addition a generator of the category $\mathrm{D}^-(\Qcoh(X))$, i.e. for $F \in \mathrm{D}^-(\Qcoh(X))$, $\RHom_X(T, F) = 0$ implies $F=0$.
We say that a tilting bundle $T$ is \textit{good} if it contains a trivial line bundle as a direct summand.
\end{defi}

For a triangulated category $\mathcal{D}$, we say that an object $E \in \mathcal{D}$ is a generator of $\mathcal{D}$ (or $E$ generates $\mathcal{D}$)
if for $F \in \mathcal{D}$, $\Hom_X(T, F[i]) = 0$ (for all $i \in \mathbb{Z}$) implies $F=0$.
We also say that $E$ is a classical generator of $\mathcal{D}$ (or $E$ classically generates $\mathcal{D}$) if the smallest thick subcategory of $\mathcal{D}$ containing $E$ is the whole category $\mathcal{D}$.
It is easy to see that a classical generator is a generator.
In some literatures (eg. \cite{Kal08, TU10}), tilting bundles are defined as a partial tilting bundle that is a generator of $\mathrm{D}^-(X) := \mathrm{D}^-(\coh(X))$.
The following lemma resolves this ambiguity of the definition of tilting bundles.

\begin{lem}
Let $X$ be a scheme that is projective over an affine scheme $S = \Spec R$, and $E$ a partial tilting bundle on $X$.
Then the following are equivalent.
\begin{enumerate}
\item[(i)] $E$ is a classical generator of the category of perfect complexes $\Perf(X)$ of $X$.
\item[(ii)] $E$ is a generator of $\mathrm{D}^-(X)$
\item[(iii)] $E$ is a generator of $\mathrm{D}^-(\Qcoh(X))$
\item[(iv)] $E$ is a generator of the unbounded derived category $\mathrm{D}(\Qcoh(X))$ of quasi-coherent sheaves.
\end{enumerate}
Here a perfect complex means a complex that is locally isomorphic to a bounded complex of vector bundles on $X$.
\end{lem}

\begin{proof}
Since $\mathrm{D}^-(X) \subset \mathrm{D}^-(\Qcoh(X)) \subset \mathrm{D}(\Qcoh(X))$,
we only have to prove that (ii) implies (i) and that (i) implies (iv).

According to Theorem 3.1.1 and Theorem 2.1.2 of \cite{BVdB03} (and also by Theorem 3.1.3 of loc. cit.),
$E$ is a classical generator of $\Perf(X)$ if and only if $E$ is a generator of $\mathrm{D}(\Qcoh(X))$.

Let us assume that $E$ is a generator of $\mathrm{D}^-(X)$ and put $\Lambda := \End_X(E)$. 
Then the following Proposition \ref{tilting equiv} implies that
there is an equivalence of categories
\[ \Psi : \D(X) \to \D(\Lambda) \]
such that $\Psi(E) = \Lambda$ (In \cite{TU10}, Toda and Uehara proved this equivalence under the assumption that the partial tilting bundle $E$
is a generator of $\mathrm{D}^-(X)$).
Let $\mathrm{K}^{\mathrm{b}}(\mathrm{proj}(\Lambda))$ be a full subcategory of $\D(\Lambda)$ consisting of complexes that are quasi-isomorphic to
bounded complexes of projective modules.
Since complexes in $\Perf(X)$ or $\mathrm{K}^{\mathrm{b}}(\mathrm{proj}(\Lambda))$ are characterized as homologically finite objects (see \cite[Proposition 1.11]{Orl06}),
the equivalence above restricts to an equivalence $\Perf(X) \simeq \mathrm{K}^{\mathrm{b}}(\mathrm{proj}(\Lambda))$.
Since the category $\mathrm{K}^{\mathrm{b}}(\mathrm{proj}(\Lambda))$ is classically generated by $\Lambda$,
the smallest thick subcategory containing $E$ should be $\Perf(X)$.
\end{proof}

We adopt the definition for tilting bundles above since in some parts of discussions in this article we need to deal with complexes of quasi-coherent sheaves.

If we find a tilting bundle on a projective scheme over an affine variety, we have an equivalence between the derived category of the scheme
and the derived category of a non-commutative ring.

\begin{prop}[{\cite[Lemma 3.3]{TU10}}] \label{tilting equiv}
Let $Y$ be a scheme that is projective over an affine scheme $X = \Spec R$.
Assume that there is a tilting bundle $T$ on $Y$.
Then, the functor
\[ \Psi := \RHom(T,-) : \mathrm{D}^-(Y) \to \mathrm{D}^-(\End_Y(T)) \]
gives an equivalence of triangulated categories.
Furthermore, this equivalence restricts an equivalence between $\D(Y)$ and $\D(\End_Y(T))$.
\end{prop}

Next we recall some basic properties of tilting bundles.
The following lemma is well-known (for example, see \cite[Lemma 3.1]{H17a}).

\begin{lem}
Let $X = \Spec R$ be a normal Gorenstein affine variety and $\phi : Y \to X$ be a crepant resolution.
Let $F$ be a coherent sheaf on $Y$ such that
\[ H^i(Y, F) = 0 = \Ext_Y^i(F, \stsh_Y) \]
for all $i > 0$.
Then the $R$-module $\phi_*F$ is maximal Cohen-Macaulay. 
\end{lem}

The following is a direct corollary of the lemma above.

\begin{cor} \label{cor CMness}
Let $X = \Spec R$ be a normal Gorenstein affine variety and $\phi : Y \to X$ be a crepant resolution.
Assume that $Y$ admits a tilting bundle $T$.
Then $\End_Y(T)$ is maximal Cohen-Macaulay as an $R$-module.

Moreover, if $T$ is a good tilting bundle, the $R$-module $\phi_*T$ is maximal Cohen-Macaulay.
\end{cor}

\begin{lem} \label{lem wemyss}
Let $X = \Spec R$ be a normal Gorenstein affine variety and $\phi : Y \to X$ be a crepant resolution.
Assume that $Y$ admits a tilting bundle $T$.
Then there is an algebra isomorphism $\End_Y(T) \simeq \End_X(\phi_*T)$.
\end{lem}

\begin{proof}
Put $M := \phi_* T$ and let $U$ be the smooth locus of $X$.
Since $\End_Y(T)$ and $\End_R(M)$ are isomorphic on $U$ and $\End_Y(T)$ is reflexive $R$-module by \ref{cor CMness},
we have that the $R$-module $\End_R(M)$ contains $\End_Y(T)$ as a direct summand, and that $\End_R(M)/\End_Y(T)$ is a
submodule of $\End_R(M)$ whose support is contained in $\Sing(X)$.

On the other hand, since $M$ is torsion free, its endomorphism ring $\End_R(M)$ is also torsion free as an $R$-module.
Thus the module $\End_R(M)/\End_Y(T)$ should be zero and hence we have
\[ \End_Y(T) \simeq \End_R(M) \]
as desired.
\end{proof}

The theory of tilting bundles has a strong relationship with the theory of non-commutative crepant resolutions, which is first introduced by Van den Bergh
\cite{VdB04b}.

\begin{defi} \rm
Let $R$ be a normal Gorenstein (commutative) algebra and $M$ be a reflexive $R$-module.
We say that the endomorphism ring $\Lambda := \End_R(M)$ is a \textit{non-commutative crepant resolution} (= NCCR) of R
(or $M$ gives an NCCR of $R$) if $\Lambda$ is maximal Cohen-Macaulay as an $R$-module and the global dimension of $\Lambda$ is finite.
\end{defi}

The relation between tilting bundles and NCCRs is given as follows.

\begin{lem} \label{from tilting to NCCR}
Let $X = \Spec R$ be a normal Gorenstein affine scheme and $\phi : Y \to X$ be a crepant resolution.
If $T$ is a good tilting bundle on $Y$, then the module $M := \phi_*T$ gives an NCCR $\End_Y(T)$ of $R$.

If the resolution $\phi$ is small (i.e. the exceptional locus $\exc(\phi)$ does not contain a divisor), then the assumption that $T$ is good is not required.
\end{lem}

\begin{proof}
If $T$ is good tilting bundle, then the push forward $M = \phi_*T$ is Cohen-Macaulay by Corollary \ref{cor CMness} and hence reflexive (\cite[Proposition 2.8]{H17a}).

If the resolution $\phi$ is small, then $M = H^0(U, T|_U)$ where $U$ is the largest common open subset of $Y$ and $X$.
Since $X \setminus U$ has codimension two in $X$, $M$ is also reflexive in this case.

Then the result follows from Proposition \ref{tilting equiv}, Corollary \ref{cor CMness} and  Lemma \ref{lem wemyss}.
\end{proof}

\subsection{Local cohomologies}

Let $X$ be a topological space and $\sh$ an abelian sheaf on $X$.
For a closed subset $Y$ of $X$, we define
\[ \Gamma_Y(\sh) = \Gamma_Y(X, \sh) := \{ s \in \Gamma(X, \sh) \mid \Supp(s) \subset Y \}. \]
This $\Gamma_Y$ defines a left exact functor
\[ \Gamma_Y : \Sh(X) \to (\mathrm{Ab}), \]
from the category of abelian sheaves on $X$ to the category of abelian groups,
and we denote the right derived functor of $\Gamma_Y$ by $H^p_Y$ or $H^p_Y(X,-)$.

Similarly, we define an abelian sheaf $\underline{\Gamma}_Y(\sh) = \underline{\Gamma}_Y(X, \sh)$ by
\[ \underline{\Gamma}_Y(\sh)(U) = \underline{\Gamma}_Y(X, \sh)(U) := \Gamma_{U \cap Y}(U, \sh|_{U}) \]
for an open subset $U \subset X$.
Then, the functor
\[ \underline{\Gamma}_Y(-) : \Sh(X) \ni \sh \mapsto \underline{\Gamma}_Y(\sh) \in \Sh(X) \]
is a left exact functor.
We denote the right derived functor of $\underline{\Gamma}_Y(-)$ by $\mathcal{H}^p_Y(-)$.

In the rest of the present subsection, we provide some basic properties of local cohomologies.

\begin{lem}[{\cite[Corollary 1.9]{Har1}}] \label{loc coho lem}
Let $X$ be a topological space, $Z \subset X$ a closed subset, $U := X \setminus Z$ the complement of $Z$, and $j : U \hookrightarrow X$ the open immersion.
Then, for any abelian sheaf $\sh$ on $X$, there are  exact sequences
\[ 0 \to \Gamma_Z(X, \sh) \to \Gamma(X, \sh) \to \Gamma(U, \sh) \to H_Z^1(X, \sh) \to H^1(X, \sh) \to H^1(U, \sh) \to \cdots \]
and
\[ 0 \to \mathcal{H}^0_Z(\sh) \to \sh \to j_*(\sh|_U) \to \mathcal{H}^1_Z(\sh) \to 0. \]
In addition, there are functorial isomorphisms
\[ R^pj_*(\sh|_U) \simeq \mathcal{H}^{p+1}_Z(\sh) \]
for $p \geq 1$.
\end{lem}

\begin{defi} \rm
Let $X$ be a locally Noetherian scheme and $Y \subset X$ a closed subset.
If $\sh$ is a coherent sheaf on $X$, then the \textit{$Y$-depth} of $\sh$ is defined by
\[ \depth_Y(\sh) := \inf_{x \in Y} \depth_{\stsh_{X,x}}(\sh_x). \]
\end{defi}

\begin{prop}[{\cite[Theorem 3.8]{Har1}}] \label{loc coho prop}
Let $X$ be a locally Noetherian scheme, $Y \subset X$ a closed subset, and $\sh$ a coherent sheaf on $X$.
Let $n$ be a non-negative integer.
Then, the following are equivalent
\begin{enumerate}
\item[(i)] $\mathcal{H}_Y^i(\sh) = 0$ for all $i < n$.
\item[(ii)] $\depth_Y(\sh) \geq n$.
\end{enumerate}
\end{prop}

\begin{cor} \label{loc coho cor1}
Let $X$ be a locally Noetherian scheme, $Y \subset X$ a closed subset, and $\sh$ a coherent sheaf on $X$.
Assume that $\depth_Y(\sh) \geq n$.
Then we have $H^i_Y(X, \sh) = 0$ for $i < n$.
\end{cor}

\begin{proof}
Let us consider a spectral sequence
\[ E_2^{p,q} = H^p(X, \mathcal{H}_Y^q(\sh)) \Rightarrow H^{p+q}_Y(X, \sh). \]
By assumption we have $E_2^{p,q} = 0$ if $q < n$.
Thus we have the result.
\end{proof}

\begin{cor} \label{loc coho cor2}
Let $X$ be a locally Noetherian scheme, $Y \subset X$ a closed subset, and $\sh$ a coherent sheaf on $X$.
Put $U := X \setminus Y$ and assume that $\depth_Y(\sh) \geq n$.
Then the canonical map
\[ H^i(X, \sh) \to H^i(U, \sh) \]
is an isomorphism for $i \leq n -2$ and
\[ H^{n-1}(X, \sh) \to H^{n-1}(U, \sh) \]
is injective.
\end{cor}

\begin{proof}
This corollary follows from Lemma \ref{loc coho lem} and Corollary \ref{loc coho cor1}.
\end{proof}

\subsection{Deformation of schemes and lift of coherent sheaves}

\begin{defi} \rm
Let $X_0$ be a scheme.
A \textit{deformation} of $X_0$ over a pointed scheme $(S, s)$ is a flat morphism $\rho : X \to (S,s)$ such that $X \otimes_{S} \Bbbk(s) \simeq X_0$.
Let $j_s : X_0 \to X$ be the closed immersion.
A deformation of $X_0$ will be denoted by
\[ (\rho : X \to (S,s), j_s : X_0 \to X). \]
If $S$ is the spectrum of a local ring $D$ and $s$ is a closed point corresponding to the unique maximal ideal,
we say that $X$ is a deformation over $D$ for short.
\end{defi}

Let us consider the category $\Art$ of local Artinian $\CC$-algebras with residue field  \nolinebreak $\CC$.
Note that, for any object $A \in \Art$, $A$ is finite dimensional as a $\CC$-vector space.

We say that a deformation $(\rho : X \to (S,s), j_s : X_0 \to X)$ of $X_0$ is \textit{infinitesimal} if $S$ is the spectrum of $A \in \Art$ and $s$ is the closed point corresponding to a unique maximal ideal $\mathfrak{m}_A \subset A$.

\begin{defi} \rm
For a scheme $X$, we define a functor
\[ \Def X : \Art^{\mathrm{op}} \to (\mathrm{Sets}) \]
by
\[ (\Def X)(A) := \{ \textrm{isom class of deformation of $X_0$ over $A$} \} \]
for $A \in \Art$.
It is easy to see that $\Def X$ actually defines a functor.
We call this functor $\Def X$ the \textit{deformation functor} of $X$ (or the local moduli functor of $X$).
\end{defi}

\begin{defi} \rm
Let $X_0$ and $Y_0$ be schemes, $\phi_0 : X_0 \to Y_0$ a morphism, and $(S, s)$ a pointed scheme.
Let $(\rho : X \to (S,s), j_s : X_0 \to X)$ and $(\pi : Y \to (S,s), i_s : Y_0 \to Y)$ be deformations over $(S,s)$.
We say that a morphism of $S$-schemes $\phi : X \to Y$ is a \textit{deformation} of $\phi_0$ if the diagram
\[ \begin{tikzcd}
X_0 \arrow[r, "j_s"] \arrow[d, "\phi_0"'] & X \arrow[d, "\phi"] \\
Y_0 \arrow[r, "i_s"] & Y
\end{tikzcd} \]
is cartesian.
\end{defi}

\begin{defi} \rm
Let $X_0$ be a scheme and $(\rho : X \to (S,s), j_s : X_0 \to X)$ a deformation of $X_0$.
Let $\sh_0$ be a coherent sheaf on $X_0$.
A \textit{lift} of $\sh_0$ to $X$ is a coherent sheaf $\sh$ on $X$ that is flat over $S$ and $j_s^*\sh \simeq \sh_0$.
\end{defi}

For more information on deformations of schemes and lifts of coherent sheaves, see \cite{Art76, Har2, Ser}.

\section{Tilting-type equivalence} \label{sect tilting-type equiv}

In this section, we provide a precise definition of (good or strict) tilting-type equivalences and study their properties.

\begin{prop} \label{def-prop tilting-type}
Let $Y$ and $Y'$ be Noetherian schemes that are projective over an affine variety $X = \Spec R$.
If there are tilting bundle $T$ and $T'$ on $Y$ and $Y'$, respectively, and an $R$-algebra isomorphism
\[ \End_Y(T) \simeq \End_{Y'}(T'), \]
then we have derived equivalences
\[ \D(Y) \simeq \D(\End_Y(T)) \simeq \D(\End_{Y'}(T')) \simeq \D(Y'). \]
\end{prop}

The following is the definition of tilting-type equivalences.

\begin{defi} \rm
Let $Y$ and $Y'$ are schemes that are projective over an affine variety $X = \Spec R$.

A \textit{tilting-type equivalence} between $\D(Y)$ and $\D(Y')$ is an equivalence of derived categories constructed as in Proposition \ref{def-prop tilting-type}.
To emphasize tilting bundles that are used to construct the equivalence, 
we also say that the tilting-type equivalence $\D(Y) \xrightarrow{\sim} \D(Y')$ is determined by $(T, T')$.

We say that a tilting-type equivalence between $\D(Y)$ and $\D(Y')$ is \textit{good} if the tilting bundles $T$ and $T'$ are good.

Assume that the morphisms $Y \to X$ and $Y' \to X$ are birational and let $U$ be the largest common open subset of $X$, $Y$, and $Y'$.
Under these conditions, we say that a tilting-type equivalence between $\D(Y)$ and $\D(Y')$ is \textit{strict} if we have
\[ T|_U \simeq T'|_U. \]
\end{defi}

\begin{lem}
Let $Y$ and $Y'$ be schemes that are projective over an affine variety $X = \Spec R$.
 Assume in addition that $Y$ and $Y'$ are Cohen-Macaulay, the morphisms $Y \to X$ and $Y' \to X$ are birational,
and they have the largest common open subscheme $U$ such that $U \neq \emptyset$ and 
\[ \codim_Y(Y \setminus U) \geq 2,  \codim_{Y'}(Y' \setminus U) \geq 2. \]
Suppose that there are tilting bundle $T$ and $T'$ on $Y$ and $Y'$, respectively, such that
\[ T|_U \simeq T'|_U. \]
Then there is a strict tilting-type equivalence
\[ \D(Y) \simeq \D(\End_Y(T)) \simeq \D(\End_{Y'}(T')) \simeq \D(Y'). \]
\end{lem}

\begin{proof}
By Proposition \ref{tilting equiv} we have equivalences
\[ \D(Y) \simeq \D(\End_Y(T)) ~ \text{and} ~ \D(Y') \simeq \D(\End_{Y'}(T')). \]
By Corollary \ref{loc coho cor2}, we have $R$-algebra isomorphisms
\[ \End_Y(T) \simeq H^0(Y, T^{\vee} \otimes T) \simeq H^0(U, T^{\vee} \otimes T) \simeq H^0(Y', T'^{\vee} \otimes T') \simeq \End_{Y'}(T'). \]
Combining the above, we have the result.
\end{proof}

Under the following nice condition, the same thing holds if $\codim_Y(Y \setminus U) = 1$.

\begin{lem} \label{lem tilting crep canoni}
Let $X = \Spec R$ be a normal Gorenstein affine variety of dimension greater than or equal to two,
and let $\phi : Y \to X$ and $\phi' : Y' \to X$ be two crepant resolutions of $X$.
Put $U := X_{\mathrm{sm}} = Y \setminus \exc(\phi) = Y' \setminus \exc(\phi')$.
Assume that there are tilting bundles $T$ and $T'$ on $Y$ and $Y'$, respectively, such that
\[ T|_U \simeq T'|_U. \]
Then there is a strict tilting-type equivalence
\[ \D(Y) \simeq \D(\End_Y(T)) \simeq \D(\End_{Y'}(T')) \simeq \D(Y'). \]
Furthermore, if $T$ and $T'$ are good, then we have an $R$-module isomorphism
\[ \phi_* T \simeq \phi'_* T'. \]
\end{lem}

\begin{proof}
By Corollary \ref{cor CMness},
we have that $\End_Y(T) \simeq \phi_*(T^{\vee} \otimes T)$ is a Cohen-Macaulay $R$-module, and hence is a reflexive $R$-module.
Therefore we have $\End_Y(T) \simeq j_*(T^{\vee} \otimes T)|_U$.

If $T$ and $T'$ are good, then $\phi_*T$ and $\phi'_*T'$ are Cohen-Macaulay and hence are reflexive.
Since $\phi_*T$ and $\phi'T'$ are isomorphic in codimension one and thus we have $\phi_*T \simeq \phi_*T'$.
\end{proof}

\begin{rem} \rm
In Lemma \ref{lem tilting crep canoni}, assume in addition that
\[ \codim_Y(Y \setminus U) \geq 2,  \codim_{Y'}(Y' \setminus U) \geq 2. \]
Then the isomorphism $\phi_*T \simeq \phi_*T'$ holds without the assumption that $T$ and $T'$ are good.

On the other hand, if $\codim_Y(Y \setminus U) = 1$, then the isomorphism $\phi_*T \simeq \phi_*T'$ does not hold without the assumption that $T$ and $T'$ are good.
Let us give an example.
Put
\[ R := \CC[x,y,z]/(x^2 + y^2 + z^2). \]
Then $X := \Spec R$ admits an $A_1$-singularity at the origin $o \in X$.
Let $\phi : Y \to X$ be the minimal resolution.
The exceptional locus $E$ of $\phi$ is an irreducible divisor of $Y$, which is a $(-2)$-curve.
More explicitly, $Y$ is the total space of a sheaf $\stsh_{\PP^1}(-2)$ on $\PP^1$ and $E$ is the zero section.
If we put $\pi : Y \to \PP^1$ be the projection and $\stsh_Y(1) := \pi^*\stsh_{\PP^1}(1)$, then it is easy to see that
\[ \stsh_Y(E) \simeq \stsh_Y(-2). \]
Thus, if we put $U := Y \setminus E = X \setminus \{o\}$, then we have there is an isomorphism
\[ \stsh_U \simeq \stsh_Y(-2)|_U. \]
Let $T := \stsh_Y \oplus \stsh_Y(1)$ and $T' := \stsh_Y(1) \oplus \stsh_Y(2)$.
An easy computation shows that $T$ and $T'$ are tilting bundles on $Y$, and the discussion above shows that we have $T|_U \simeq T'|_U$.
On the other hand, there is an exact sequence on $Y$
\[ 0 \to \stsh_Y(2) \oplus \stsh_Y(1) \xrightarrow{E \oplus \id} \stsh_Y \oplus \stsh_Y(1) \to \stsh_E \to 0. \]
Applying the functor $\phi_*$ to this sequence, we have a short exact sequence
\[ 0 \to \phi_*T' \to \phi_*T \to \Bbbk(o) \to 0, \]
where $\Bbbk(o)$ is the residue field at the origin $o \in X$.
In particular, $\phi_*T'$ is not reflexive and hence we have $\phi_*T' \neq \phi_*T$.
\end{rem}

In almost all known examples, a tilting bundle (on a crepant resolution) contains a line bundle as a direct summand.
The following lemma suggests that the assumption that a tilting-type equivalence is good is not strong if the resolutions are small.

\begin{lem}
Let $X = \Spec R$ be a normal Gorenstein affine variety and assume that $X$ admits two small resolutions $\phi : Y \to X$ and $\phi' : Y' \to X$.
If there exists a strict tilting-type equivalence between $\D(Y)$ and $\D(Y')$ determined by $(T, T')$
such that $T$ contain a line bundle $L$ on $Y$ as a direct summand,
then there is a good tilting-type equivalence between $\D(Y)$ and $\D(Y')$.
\end{lem}

\begin{proof}
Put $U := Y \setminus \exc(\phi) = Y' \setminus \exc(\phi')$ and let $j : U \hookrightarrow Y'$ be the open immersion.
Then, since $\phi'$ is small, $L' := j'_*(L|_U)$ is a divisorial sheaf.
Moreover, since $Y'$ is smooth, $L'$ is a line bundle on $Y'$.
Note that $L'$ is contained in $T'$ as a direct summand.
Indeed, since $L'|_U = L|_U$ is a direct summand of $T'|_U = T|_U$, the inclusion $L'|_U \hookrightarrow T'|_U$ has a splitting morphism $T'|_U \to L'|_U$.
Since $Y' \setminus U$ has codimension greater than or equal to two, this splitting extends to a splitting of the inclusion $L' \hookrightarrow T'$.
Therefore we have that
\[ (T \otimes L^{\vee})|_U \simeq (T' \otimes L'^{\vee})|_U \]
and that $T \otimes L^{\vee}$ and $T' \otimes L'^{\vee}$ are good tilting bundles.
Then, by Lemma \ref{lem tilting crep canoni}, we have the result.
\end{proof}

Next we give the definition of the lift of a tilting-type equivalence.

\begin{defi} \label{def lift tilting-type} \rm
Let $X_0 = \Spec R_0$ be an affine scheme, and let $\phi_0 : Y_0 \to X_0$ and $\phi_0' : Y_0' \to X_0$ be two projective morphisms.
Let us consider deformations of varieties and morphisms above
\[ \begin{tikzcd}
Y \arrow[r, "\phi"] \arrow[rd] & X \arrow[d] & Y' \arrow[l, "\phi'"'] \arrow[ld] \\
& (S, s) &
\end{tikzcd} \]
over an pointed scheme $(S,s)$, such that $X = \Spec R$ is affine and two morphisms $\phi$ and $\phi'$ are projective.
Let $\Phi_0 : \D(Y_0) \xrightarrow{\sim} \D(Y_0')$ be a tilting-type equivalence determined by $(T_0, T_0')$.

Under the set-up above, we say that a tilting-type equivalence $\Phi : \D(Y) \xrightarrow{\sim} \D(Y')$ is a \textit{lift} of $\Phi_0$
if $\Phi$ is determined by $(T, T')$ such that $T$ (resp. $T'$) is a lift of $T_0$ (resp. $T_0'$) to $Y$ (resp. $Y'$),
and if the algebra isomorphism $\End_Y(T) \simeq \End_Y(T')$ coincides with the algebra isomorphism $\End_{Y_0}(T_0) \simeq \End_{Y'_0}(T'_0)$
after we restrict it to $X_0$.
\end{defi}

The following is an easy observation on a lift of a tilting-type equivalence.

\begin{lem}
Under the same condition as in Definition \ref{def lift tilting-type}, let $j_s : Y_0 \hookrightarrow Y$ and $j_s' : Y'_0 \hookrightarrow Y'$ be closed immersions
associated to the deformations.
Let $\Phi : \D(Y) \xrightarrow{\sim} \D(Y')$ be a tilting-type equivalence that is a lift of a tilting-type equivalence $\Phi_0 : \D(Y_0) \xrightarrow{\sim} \D(Y'_0)$.
Then the following diagram of functors commutes
\[ \begin{tikzcd}
\D(Y_0) \arrow[d, "\Phi_0"'] \arrow[r, "(j_s)_*"] & \D(Y) \arrow[d, "\Phi"] \\
\D(Y'_0) \arrow[r, "(j'_s)_*"] & \D(Y').
\end{tikzcd} \]
\end{lem}

\begin{proof}
By adjunction and the construction of the equivalences, we have
\begin{align*}
\RHom_{Y'}(T', (j'_s)_*(\Phi_0(F))) &\simeq \RHom_{Y'_0}(T_0', \Phi_0(F)) \\
&\simeq \RHom_{Y_0}(T_0, F) \\
&\simeq \RHom_{Y}(T, (j_s)_*(F)) \\
&\simeq \RHom_{Y'}(T', \Phi((j_s)_*(F)))
\end{align*}
for all $F \in \D(Y_0)$.
Since the functor
\[ \RHom_{Y'}(T', -) : \D(Y') \to \D(\End_{Y'}(T')) \]
gives an equivalence, we have a functorial isomorphism
\[ (j'_s)_*(\Phi_0(F)) \simeq \Phi((j_s)_*(F)). \]
This shows the result.
\end{proof}

\section{Infinitesimal deformation of small resolutions} \label{sect infinitesimal}

Let $X_0 = \Spec R_0$ be a normal Gorenstein affine variety and $\phi_0 : Y_0 \to X_0$ a crepant resolution.
Throughout this section we always assume that
\[ \codim_{Y_0} \exc(\phi_0) \geq 3. \]
Under this assumption, the deformation theory behaves very well.
First, we observe the following proposition.

\begin{prop} \label{prop 3-1}
Under the conditions above, there is a functor isomorphism
\[ \Def Y_0 \simeq \Def X_0. \]
\end{prop} 

This proposition follows immediately from the following lemma.

\begin{lem}
Let $Z_0$ be a variety and $j : U_0 \hookrightarrow Z_0$ an open subset.
Assume one of the following conditions.
\begin{enumerate}
\item[(1)] $Z_0$ is affine, Cohen-Macaulay, $\dim Z_0 \geq 3$, and $U_0 = (Z_0)_{\mathrm{sm}}$.
\item[(2)] $Z_0$ is smooth and $\codim_{Z_0}(Z_0 \setminus U_0) \geq 3$.
\end{enumerate}
Then, the restriction
\[ \Def Z_0 \to \Def U_0 \]
gives an isomorphism of functors.
\end{lem}

\begin{proof}
The case (1) is proved in \cite[Proposition 9.2]{Art76}.
Let us prove the lemma under the condition (2).
Let $A$ be a local Artinian algebra with residue field $\CC$ and $U$ a deformation of $U_0$ over $A$.
We will show that $\Def(Z_0)(A) = \Def(U_0)(A)$ by an induction on the dimension of $A$ as a $\CC$-vector space.
Let
\[ e : 0 \to (t) \to A \to A' \to 0 \]
be a small extension and $U' := U \otimes_A A'$.
By induction hypothesis, there is the unique deformation $Z'$ of $Z_0$ over $A'$ such that $Z'|_{U_0} = U'$.

By assumption, Lemma \ref{loc coho lem} and Proposition \ref{loc coho prop}, we have 
an isomorphism
\[ H^1(Z_0, \Theta_{Z_0}) \xrightarrow{\sim} H^1(U, \Theta_{U_0}) \]
and an injective map
\[ H^2(Z_0, \Theta_{Z_0}) \hookrightarrow H^2(U, \Theta_{U_0}). \]
Note that the second map is compatible with the obstruction map by construction.
Since there is a lifting $U$ of $U'$ to $A$, the obstruction map sends $e$ to $0$, and hence the set
\[ \{ \text{isom class of lifts of $Z'$ to $A$} \} \]
is non-empty.
By deformation theory, the first cohomology group $H^1(Z_0, \Theta_{Z_0}) = H^1(U, \Theta_{U_0})$ acts on the sets
\[ \{ \text{isom class of lifts of $Z'$ to $A$} \} ~ \text{and} ~ \{ \text{isom class of lifts of $U'$ to $A$} \} \]
transitively, and the restriction map
\[  \{ \text{isom class of lifts of $Z'$ to $A$} \} \to \{ \text{isom class of lifts of $U'$ to $A$} \}  \]
is compatible with these actions. Thus, the above restriction map is bijective and hence there is a unique lift of $Z$ of $Z'$ to $A$, which satisfies $Z|_{U_0} = U$.

Let $Z^1$ and $Z^2$ be two deformation of $Z_0$ over $A$ such that $Z^i|_{U_0} = U$ $(i=1,2)$.
If we set $Z'^i := Z^i \otimes_A A'$, then $Z'^i$ is the extension of $U'$ for $i = 1, 2$ and hence we have $Z'^1 = Z'^2$.
Thus, the above argument shows that we have $Z^1 = Z^2$.
This shows the result.
\end{proof}

Let $A$ be a local Artinian algebra with residue field $\CC$.
Take an element
\[ \xi \in (\Def Y_0)(A) = (\Def X_0)(A), \]
and let $Y$ and $X$ be deformations of $Y_0$ and $X_0$, respectively, over $A$ that correspond to $\xi$.

It is easy to observe the following.
\begin{enumerate}
\item[(1)] $X$ and $Y$ are of finite type over $A$ (or $\CC$).
\item[(2)] The inclusions $X_0 \subset X$ and $Y_0 \subset Y$ are homeomorphisms.
\end{enumerate}
For example, (1) is proved in \cite{Art76}.

\begin{lem}
$X$ and $Y$ are Cohen-Macaulay schemes.
Furthermore, if the local Artinian algebra $A$ is Gorenstein, then so are $X$ and $Y$.
\end{lem}

\begin{proof}
Note that Artinian algebras are zero-dimensional and hence are Cohen-Macaulay.
In addition, for any point $x \in X$ (resp. $y \in Y$), the homomorphism $A \to \stsh_{X,x}$ (resp. $A \to \stsh_{Y, y}$) is automatically local.
Then, the statement follows from \cite[Corollary of Theorem 23.3 and Theorem 23.4]{Mat}.
\end{proof}

\begin{lem}
Let us consider the map $\phi_0 : Y \to X$ of topological spaces.
Then the direct image sheaf $(\phi_0)_*\stsh_Y$ is isomorphic to $\stsh_X$ as sheaves of $A$-algebras.
In particular, there is a morphism of $A$-schemes $\phi : Y \to X$ whose pull back by $\Spec \CC \to \Spec A$ is $\phi_0$.
\end{lem}

\begin{proof}
Let $j : U \to X$ and $i : U \to Y$ be open immersions.
Since $X$ and $Y$ are Cohen-Macaulay, we have isomorphisms of sheaves of rings
\[ (\phi_0)_*\stsh_Y \simeq (\phi_o)_*i_*\stsh_U \simeq j_*\stsh_U \simeq \stsh_X. \]
It is clear that all isomorphisms are $A$-linear.
\end{proof}

\begin{lem} \label{inf defo proj}
The morphism $\phi : Y \to X$ is projective.
\end{lem}

\begin{proof}
First, we note that the morphism $\phi$ is proper.
Thus, it is enough to show that there is a $\phi$-ample line bundle.

Let $L_0$ be a $\phi_0$-ample line bundle.
Since
\[ \Ext_{Y_0}^p(L_0, L_0) \simeq H^p(Y_0, \stsh_{Y_0}) \simeq R^p{\phi_0}_*(\stsh_{Y_0}) \]
and $\phi_0$ is a rational resolution of $X_0$, the line bundle $L_0$ is a partial tilting bundle on $Y_0$.
Thus due to Proposition \ref{Kar1} below,  $L_0$ has the unique lifting $L$ on $Y$, which is invertible.

We show that this line bundle $L$ on $Y$ is $\phi$-ample.
Since $X_0 \hookrightarrow X$ is a homeomorphism, we can regard a closed point $x \in X $ as a closed point of $X_0$,
and we have $\phi^{-1}(x) = \phi_0^{-1}(x)$.
Thus, we have $L|_{\phi^{-1}(x)} \simeq L_0|_{\phi_0^{-1}(x)}$ and hence $L|_{\phi^{-1}(x)}$ is absolutely ample.
Then, by \cite[Theorem 1.2.17 and Remark 1.2.18]{Laz}, we have that $L$ is $\phi$-ample.
\end{proof}

Next we discuss tilting bundles on $Y$.
First we recall the following result due to Karmazyn.

\begin{prop}[{\cite[Theorem 3.4]{Kar15}}] \label{Kar1}
Let $\pi_0 : Z_0 \to \Spec S_0$ be a projective morphisms of Noetherian schemes.
Let $(A, \mathfrak{m})$ be a local Artinian algebra with residue field $\CC$,
and let $Z$ and $\Spec S$ be deformations of $Z_0$ and $\Spec S_0$, respectively, over $A$.
Assume that there is an $A$-morphism $\pi : Z \to \Spec S$ such that $\pi \otimes_A A/\mathfrak{m} = \pi_0$.

Then, for any partial tilting bundle $T_0$ on $Z_0$, there is the unique lifting $T$ of $T_0$ on $Z$, which is partial tilting.
Moreover, if $T_0$ is tilting, then so is $T$.
\end{prop}

The following lemma is a certain variation of the proposition above.

\begin{prop} \label{prop 3-key1}
With the same condition as in Proposition \ref{Kar1}, assume in addition that $Z_0$ is a Cohen-Macaulay variety of dimension greater than or equal to three.
Let $T_0$ be a partial tilting bundle on $Z_0$, and let $U_0 \subset Z_0$ be an open subscheme of $Z_0$.
Put $U := Z|_{U_0}$ and assume that
\[ \codim_{Z_0}(Z_0 \setminus U_0) \geq 3. \]
Then, the bundle $T_0|_{U_0}$ on $U_0$ lifts uniquely to a bundle on $U$.
\end{prop}

\begin{proof}
Let $\mathfrak{m}$ be the maximal ideal of $A$ and put
\[ A_n := A/\mathfrak{m}^{n+1}, ~ Z_n := Z \otimes_A A_n, ~ \text{and} ~ U_n := U \otimes_A A_n \]
for $n \geq 0$.
Since $A$ is Artinian, we have $A = A_n$ for sufficiently large $n$.
By Proposition \ref{Kar1}, there is the unique lifting $T_n$ of $T_0$ on $Z_n$, which is partial tilting.
The existence of a unique lifting of $T_0|_{U_0}$ on $U_n$ follows from this result.

We prove the uniqueness by an induction on $n$.
Put $T'_n := T_n|_{U_n}$.
Then the set
\[ \{ \text{isom class of liftings of $T'_n$ on $U_{n+1}$} \} \]
is non-empty and is a torsor under the action of $H^1(U_0, \EEnd_{U_0}(T'_0)) \otimes_{\CC} \mathfrak{m}^{n}/\mathfrak{m}^{n+1}$.

By adjunction, we have an isomorphism
\[ H^1(U_0, \EEnd_{U_0}(T'_0)) \simeq H^1(Z_0, \EEnd_{Z_0}(T_0) \otimes_{Z_0} Rj_*\stsh_{U_0}), \]
where $j : U_0 \to Z_0$ is an open immersion.
Thus, there is a spectral sequence
\[ E_2^{p, q} := H^p(Z_0, \EEnd_{Z_0}(T_0) \otimes_{Z_0} R^qj_*\stsh_{U_0}) \Rightarrow H^{p+q}(U_0, \EEnd_{U_0}(T'_0)). \]
Since $Z_0$ is Cohen-Macaulay and $T_0$ is partial tilting on $Z_0$, we have
\[ E_2^{p,0} = \Ext_{Z_0}^p(T_0, T_0) = 0 \]
for $p > 0$.
By assumption, we have $R^1j_*\stsh_{U_0} = 0$ and hence we have $E_2^{p, 1} = 0$ for all $p$.
In particular, we have $E_2^{p,q} = 0$ for $p, q$ with $p+q = 1$, and hence we have
\[ H^{1}(U_0, \EEnd_{U_0}(T'_0)) = 0. \]
Thus, the lifting of $T'_n$ on $U_{n+1}$ is unique.
\end{proof}

\begin{rem} \rm
Put $E_0 := Z_0 \setminus U_0$ and assume that $\codim_{Z_0}(E_0) = 2$.
In this case, analyzing the proof above, we notice that the vanishing of $H^0(Z_0, \EEnd_{Z_0}(T_0) \otimes R^1j_* \stsh_{U_0}) = 0$ is sufficient.
If $Z_0$ and $E_0$ are smooth, one can easily show that there is an isomorphism
\[ H^0(Z_0, \EEnd_{Z_0}(T_0) \otimes R^1j_* \stsh_{U_0}) \simeq \bigoplus_{k \geq 0} H^0(E_0, \EEnd_{E_0}(T|_{E_0}) \otimes \Sym^k N_{E_0/Z_0} \otimes \det(N_{E_0/Z_0})), \]
where $N_{E_0/Z_0}$ is the normal bundle of $E_0 \subset Z_0$.
\end{rem}

\begin{thm} \label{inf def main thm}
Let $X_0$ be a Gorenstein normal affine variety, $\phi_0 : Y_0 \to X_0$ and $\phi'_0 : Y'_0 \to X_0$ two crepant resolutions of $X_0$.
Assume that 
\[ \codim_{Y_0}(\exc(\phi_0)) \geq 3 ~  \text{and} ~ \codim_{Y'_0}(\exc(\phi'_0)) \geq 3. \]
Then,
\begin{enumerate}
\item[(1)] there is a functor isomorphism $\Def Y_0 \simeq \Def Y'_0$.
\item[(2)] Let $A$ be a local Artinian algebra with residue field $\CC$. Let $Y$ and $Y'$ be deformations of $Y_0$ and $Y'_0$, respectively, over $A$ that correspond to an element $\xi \in (\Def Y_0)(A) \simeq (\Def Y'_0)(A)$.
Then, any strict tilting-type equivalence $\D(Y_0) \xrightarrow{\sim} \D(Y'_0)$ lifts to a strict tilting-type equivalence $\D(Y) \xrightarrow{\sim} \D(Y')$.
\end{enumerate}
\end{thm}

\begin{proof}
The first assertion follows from Proposition \ref{prop 3-1}.

Put $U_0 := Y_0 \setminus \exc(\phi_0) = (X_0)_{\mathrm{sm}} = Y'_0 \setminus \exc(\phi'_0)$
and set $U := Y|_{U_0} = Y'|_{U_0}$.
Let $T_0$ (resp. $T'_0$) be a tilting bundle on $Y_0$ (resp. $Y'_0$) such that $T_0|_{U_0} = T'_0|_{U_0}$,
and $T$ (resp. $T'$) the unique lifting of $T_0$ (resp. $T'_0$) on $Y$ (resp. $Y'$).
Then, by Proposition \ref{prop 3-key1}, we have $T|_{U} = T'|_{U}$.
Since $Y$ and $Y'$ are projective over an affine variety, from Proposition \ref{tilting equiv}, we have equivalences of categories
\[ \D(Y) \simeq \D(\End_{Y}(T)) \simeq \D(\End_{Y'}(T')) \simeq \D(Y'). \]
This is what we want.
\end{proof}

\begin{rem} \rm
Compare our result with the result of Toda \cite[Theorem 7.1]{Tod09}, where he proved a similar result for an equivalence for flops given as a Fourier-Mukai transform.
\end{rem}

\section{Deformation of crepant resolutions over a complete local ring} \label{sect complete local}

The goal of the present section is to prove the following theorem.

\begin{thm} \label{thm formal def}
Let $X_0 := \Spec R_0$ be a normal Gorenstein affine variety, and $\phi_0 : Y_0 \to X_0$ and $\phi_0' : Y'_0 \to X_0$ two crepant resolution of $X_0$.
Let $(D, d)$ be a Cohen-Macaulay complete local ring, and let a diagram
\[ \begin{tikzcd}
\mathcal{Y} \arrow[rd] \arrow[r, "\varphi"] & \mathcal{X} \arrow[d] & \mathcal{Y}' \arrow[l, "\varphi'"'] \arrow[ld] \\
& \Spec D &
\end{tikzcd} \]
be a deformation of varieties and morphisms above over $(D,d)$,
where $\mathcal{X} = \Spec \mathcal{R}$ is an affine scheme, and $\varphi$ and $\varphi'$ are projective morphisms.
Assume that 
\[ \codim_X \Sing(X_0) \geq 3. \]
Then any good and strict tilting-type equivalence for $Y_0$ and $Y'_0$ lifts to a good tilting type equivalence for $\mathcal{Y}$ and $\mathcal{Y}'$.

In addition, if the lift is determined by tilting bundles $(\mathcal{T}, \mathcal{T'})$, then we have $\varphi_* \mathcal{T} \simeq \varphi'_* \mathcal{T}'$.
\end{thm}

\begin{rem} \rm
The assumption $\codim_{X_0} \Sing(X_0) \geq 3$ is essential and there is a counter-example of this theorem if we remove this assumption
(see Section \ref{sect Eg counter}).
\end{rem}

\subsection{Preliminaries}

Recall that a finitely generated $S$-module $M$ is said to be \textit{rigid} if $\Ext^1_{S}(M, M) = 0$, and to be \textit{modifying} if $M$ is reflexive and $\End_S(M)$ is a maximal Cohen-Macaulay $S$-module.
By definition, if $M$ gives an NCCR of $S$ then $M$ is modifying.

\begin{lem}
Let $Z := \Spec S$ be a Cohen-Macaulay affine variety and $M$ is a modifying module over $S$ that is locally free in codimension two.
Assume that $\dim S \geq 3$ and $\codim \Sing(\Spec S) \geq 3$.
Then the module $M$ is rigid.
\end{lem}

\begin{proof}
Since $M$ is reflexive, it satisfies $(S_2)$ condition.
In addition, since $\End_S(M)$ is maximal Cohen-Macaulay and $\dim S \geq 3$, $\End_S(M)$ satisfies $(S_3)$ condition.
Then the result follows from \cite[Lemma 2.3]{Dao10}.
\end{proof}

\begin{cor}
Let $Z = \Spec S$ be a normal Gorenstein affine variety of dimension greater than or equal to three and $\psi : W \to Z$ be a crepant resolution.
Assume that $\codim_Z \Sing(Z) \geq 3$, and that $W$ admits a good tilting bundle $T$.
Then the $S$-module $M := \psi_* T$ is rigid.

If the resolution $\psi : W \to Z$ is small, we can remove the assumption that $T$ is good.
\end{cor}

\begin{proof}
Since we assumed that $\codim_Z \Sing(Z) \geq 3$, the module $M$ is locally free in codimension two.
In addition, since the tilting bundle $T$ is good, the module $M$ is Cohen-Macaulay and hence reflexive.
Note that if the resolution $\psi : W \to Z$ is small, we have that $M$ is reflexive without the assumption that $T$ is good.
Since $\End_W(T) \simeq \End_Z(M)$ is Cohen-Macaulay, the module $M$ is modifying.
Thus we have the result.
\end{proof}

The following proposition should be well-known, but we provide a sketch of the proof here because the author has no reference for this proposition.

\begin{prop}
Let $S_0$ be a $\CC$-algebra and $M_0$ a finitely generated $S_0$-module.
Let $A$ a local Artinian algebra with residue field $\CC$ and
\[ 0 \to J \to A' \to A \to 0 \]
an extension of $A$ such that $J^2 = 0$.
Let $S$ be a deformation of $S_0$ over $A$, $M$ a lift of $M_0$ over $S$, and $S'$ a lift of $S$ over $A'$.
Assume that there exists a lift $M'$ of $M$ over $S'$.
Then the lift of $M$ over $S'$ is a torsor of $\Ext^1_{S_0}(M_0, M_0) \otimes_{\CC} J$.
\end{prop}

\begin{proof}
Let $M'_1$ and $M'_2$ be two lifts of $M$ over $S'$.
Then there exist a free $S'$-module $P'$ and surjective morphisms $P' \to M'_1$ and $P' \to M'_2$ whose restriction to $S$ coincide with each other.
Let us consider the following diagram
\[ \begin{tikzcd}
 & 0 \arrow[d] & 0 \arrow[d] & 0 \arrow[d] & \\
0 \arrow[r] & J \otimes_{\CC} N_0 \arrow[r] \arrow[d] & N_i' \arrow[r] \arrow[d] & N \arrow[d] \arrow[r] & 0 \\
0 \arrow[r] & J \otimes_{\CC} P_0 \arrow[r] \arrow[d] & P' \arrow[r] \arrow[d] & P \arrow[d] \arrow[r] & 0 \\
0 \arrow[r] & J \otimes_{\CC} M_0 \arrow[r] \arrow[d] & M'_i \arrow[r] \arrow[d] & M \arrow[d] \arrow[r] & 0 \\
& 0 & 0 & 0 &
\end{tikzcd} \]
Take an element $x \in N$ and choose its lifts $x'_1 \in N_1'$ and $x'_2 \in N'_2$.
Then we can consider the difference $x'_1 - x_2'$ in $P'$ and we have $x'_1 - x'_2 \in J \otimes_{\CC} P_0$.
Let $y \in J \otimes_{\CC} M_0$ be the image of $x'_1 - x'_2 \in J \otimes_{\CC} P_0$.
Then the correspondence $N \ni x \mapsto y \in J \otimes_{\CC} M_0$ gives an $S_0$-module homomorphism $\gamma : N_0 \to J \otimes_{\CC} M_0$.
It is easy to see that $N'_1 = N'_2$ as submodules of $P'$ if and only if $\gamma = 0$,
and the similar argument shows that the ambiguity of the choice of a surjective morphism $P' \to M'$ is resolved by $\Hom_{S_0}(P_0, J \otimes_{\CC} M_0)$.
Thus we have the result.
\end{proof}

Under the set-up of Theorem \ref{thm formal def}, put $(D_n, d_n) := (D/d^{n+1}, d/d^{n+1})$,
\[ X_n := \mathcal{X} \otimes_D D_n, ~ Y_n := \mathcal{Y} \otimes_D D_n, ~ Y_n' := \mathcal{Y} \otimes_D D_n, \]
and let
\[ \phi_n := \phi \otimes_D D_n : Y_n \to X_n, ~ \phi'_n := \phi' \otimes_D D_n : Y'_n \to X_n \]
be the projections.

\begin{lem} \label{good flat}
Let $T_0$ be a good tilting bundle on $Y_0$ and $T_n$ the lift of $T_0$ to $Y_n$.
Then a sheaf $M_n := (\phi_n)_* T_n$ on $X_n$ is a lift of $M_0 := (\phi_0)_* T_0$ to $X_n$.
\end{lem}

\begin{proof}
Note that $T_n$ is also good and hence we have $M_n \simeq R{\phi_n}_* T_n$ and $M_0 \simeq R{\phi_0}_* T_0$.
Then the result follows from \cite[Corollary 2.11]{Kar15}.
\end{proof}





\subsection{Proof of the theorem}

\begin{proof}[Proof of Theorem \ref{thm formal def}]

Let $U_0$ be the common open subscheme of $X_0$, $Y_0$ and $Y'_0$.
Let $T_0$ and $T'_0$ be good tilting bundles on  $Y_0$ and $Y'_0$, respectively,
such that $T_0|_{U_0} \simeq T'_0|_{U_0}$.
Put $U_n := X_n|_{U_0}$.

There exist good tilting bundles $T_n$ and $T'_n$ on $Y_n$ and $Y'_n$, respectively,
such that $T_n$ (resp. $T'_n$) is a lift of $T_0$ (resp. $T'_0$).
Then $(\phi_n)_*T_n$ and $(\phi'_n)_*T'_n$ are two lifts of $M_0 = H^0(T_0) = H^0(T_0')$ by Lemma \ref{good flat}.
By assumptions the module $M_0$ is rigid and hence we have
\[ (\phi_n)_*T_n \simeq (\phi'_n)_*T'_n. \]
In particular, we have
\[ T_n|_{U_n} \simeq T'_n|_{U_n}. \]
Next we show that we have an algebra isomorphism
\[ \End_{Y_n}(T_n) \simeq \End_{Y'_n}(T'_n). \]
As in the proof of Lemma \ref{good flat}, we have that 
$H^0(Y_n, T_n^{\vee} \otimes T_n)$ and $H^0(Y'_n, T'^{\vee}_n \otimes T'_n)$ are coherent sheaves on $X_n$ that are flat over $D_n$.
On the other hand, we have
\begin{align*}
&H^0(Y_n, T^{\vee}_n \otimes T_n) \otimes_{D_n} D_n/d_n \simeq H^0(Y_0, T_0^{\vee} \otimes T_0) ~ \text{and} \\
&H^0(Y'_n, T'^{\vee}_n \otimes T'_n) \otimes_{D_n} D_n/d_n \simeq H^0(Y_0, T'^{\vee}_0 \otimes T'_0)
\end{align*}
are Cohen-Macaulay modules over $X_0$.
Therefore the flat extension theorem \cite[Theorem 23.3]{Mat} implies that
$H^0(Y_n, T_n^{\vee} \otimes T_n)$ and $H^0(Y'_n, T'^{\vee}_n \otimes T'_n)$ are Cohen-Macaulay modules over $X_n$.
Thus we have algebra isomorphisms
\[ H^0(Y_n, T_n^{\vee} \otimes T_n) \simeq H^0(U_n, T_n^{\vee} \otimes T_n) \simeq H^0(U_n, T'^{\vee}_n \otimes T'_n) \simeq H^0(Y'_n, T'^{\vee}_n \otimes T'_n). \]
Finally unwinding Grothendieck's existence theorem \cite[Tag 088F]{Sta} and the result of Karmazyn \cite[Lemma 3.3]{Kar15} imply that there exist good tilting bundles $\mathcal{T}$ and $\mathcal{T}'$ on $\mathcal{Y}$ and $\mathcal{Y}'$,
 respectively, such that
 \[ \End_{\mathcal{Y}}(\mathcal{T}) \simeq \lim \End_{Y_n}(T_n) \simeq \lim \End_{Y'_n}(T'_n) \simeq \End_{\mathcal{Y'}}(\mathcal{T}'). \]
 Note that we have $(\varphi_*\mathcal{T})|_{X_n} \simeq (\phi_n)_* T_n \simeq (\phi'_n)_* T'_n \simeq (\varphi'_*\mathcal{T}')|_{X_n}$
 by construction.
 This implies that we have an isomorphism $\varphi_*\mathcal{T} \simeq \varphi'_*\mathcal{T}'$ (see \cite[Tag 087W]{Sta}).
 This shows the result.
 
 In addition, if we put $\mathcal{M} := \varphi_* \mathcal{T} \simeq \varphi'_* \mathcal{T}$, we also have isomorphisms
 \[ \End_{\mathcal{Y}} (\mathcal{T}) \simeq \lim \End_{Y_n}(T_n) \simeq \lim \End_{X_n}((\phi_n)_*T_n) = \End_{\mathcal{R}}(\mathcal{M}), \]
 which also follow from \cite[Tag 087W]{Sta}.
\end{proof}

The following corollary is a direct consequence of the proof of Theorem \ref{thm formal def}.

\begin{cor} \label{cor thm form defo}
Under the same set-up of Theorem \ref{thm formal def}, assume in addition that
\[ \codim_Y(\exc(\phi_0)) \geq 3. \]
Then any strict tilting-type equivalence for $Y_0$ and $Y'_0$ lifts to a tilting type equivalence for $\mathcal{Y}$ and $\mathcal{Y}'$.
\end{cor}

\section{Deformation with an action of $\Gm$} \label{sect equivariant}

\subsection{$\Gm$-action}

First of all, we recall some basic definitions and properties of actions of a group $\Gm$ on schemes.

\begin{defi} \label{def good action} \rm
We say that a $\Gm$-action on an affine scheme $\Spec R$ is \textit{good} if there is a unique $\Gm$-fixed closed point corresponding to a maximal ideal $\mathfrak{m}$,
such that $\Gm$ acts on $\mathfrak{m}$ by strictly positive weights. 
\end{defi}

The advantage to consider good $\Gm$-actions is that we can use the following useful theorem.

\begin{thm}[\cite{Kal08}, Theorem 1.8 (ii)] \label{Kal08 thm1.8}
Let $Y$ be a scheme that is projective over an affine variety $X = \Spec R$,
and assume that $X$ admits a good $\Gm$-action that lifts to a $\Gm$-action on $Y$.
Let $\widehat{X}$ be a completion of $X = \Spec R$ with respect to the maximal ideal $\mathfrak{m} \subset R$ that corresponds to a unique fixed point $x \in X$.

Then any tilting bundle on $\widehat{Y} := Y \times_X \widehat{X}$ is obtained by a pull-back of a $\Gm$-equivariant tilting bundle on $Y$.
\end{thm}

In relation to the theorem above, see also \cite[Appendix A]{Nam08} and \cite[Proposition 5.1]{Kar15}.

\begin{lem}[\cite{Kal08}, Lemma 5.3] \label{Kal08 lem5.3}
Let $R$ be a $\CC$-algebra of finite type such that the corresponding affine scheme $\Spec R$ admits a good $\Gm$-action.
Let $\mathfrak{m} \subset R$ be the maximal ideal that corresponds to the unique fixed point of $\Spec R$, and
$\widehat{R}$ the completion of $R$ with respect to the maximal ideal $\mathfrak{m}$.

Then, the $\mathfrak{m}$-adic completion functor gives an equivalence between the category of finitely generated $\Gm$-equivariant $R$-modules and the category of complete Noetherian $\Gm$-equivariant $\widehat{R}$-modules.
\end{lem}

\subsection{Result}

Let $X_0 := \Spec R_0$ be a normal Gorenstein affine variety, and $\phi_0 : Y_0 \to X_0$ and $\phi_0' : Y'_0 \to X_0$ two crepant resolutions of $X_0$.
Let $(\Spec D, d)$ be a pointed affine variety, and let the diagram
\[ \begin{tikzcd}
Y \arrow[rd] \arrow[r, "\phi"] & X \arrow[d] & Y' \arrow[l, "\phi'"'] \arrow[ld] \\
& \Spec D &
\end{tikzcd} \]
be a deformation of varieties and morphisms above over $(\Spec D,d)$
such that the varieties $Y$, $Y'$ and $X$ are $\Gm$-variety and the morphisms $\phi$ and $\phi'$ are projective and $\Gm$-equivariant.
Assume that $X = \Spec R$ is an affine \textit{variety} and that the $\Gm$-action on $X$ is good whose fixed point $x \in X$ is a point over $d \in \Spec D$.

In this subsection, we prove the following theorem.

\begin{thm} \label{main thm equiv}
Under the set-up above, assume in addition that the inequality $\codim_{X_0} \Sing(X_0) \geq 3$ holds.
Then any good and strict tilting-type equivalence between $\D(Y_0)$ and $\D(Y'_0)$ lifts to a tilting type equivalence for $\D(Y)$ and $\D(Y')$.

In addition, if the lift is determined by tilting bundles $(T, T')$, then we have $\phi_*T \simeq \phi'_*T'$.
\end{thm}

\begin{rem} \rm
Again we remark that the assumption $\codim_{X_0} \Sing(X_0) \geq 3$ is essential for this theorem
(see Section \ref{sect Eg counter}).
\end{rem}

\begin{proof}[Proof of Theorem \ref{main thm equiv}]
In the following, we provide a proof of the theorem that consists of three steps.

\noindent
\textbf{Step 1. Construction of tilting bundles}.

Put $(D_n, d_n) := (D/d^{n+1}, d/d^{n+1})$, $\mathcal{D} := \lim D_n$,
$\mathcal{X} := X \otimes_D \mathcal{D}$, $\mathcal{Y} := Y \otimes_D \mathcal{D}$, and $\mathcal{Y}' := Y' \otimes_D \mathcal{D}$.
Let
\[ \varphi : \mathcal{Y} \to \mathcal{X} ~ \text{and} ~ \varphi' : \mathcal{Y}' \to \mathcal{X} \]
be the projection.
Then we can apply Theorem  \ref{thm formal def} to the diagram
\[ \begin{tikzcd}
\mathcal{Y} \arrow[r, "\varphi"] \arrow[rd] & \mathcal{X} \arrow[d] & \mathcal{Y'} \arrow[l, "\varphi'"'] \arrow[ld] \\
& \Spec \mathcal{D} &
\end{tikzcd} \]
and have a tilting bundle $\mathcal{T}$ (resp. $\mathcal{T}'$) on $\mathcal{Y}$ (resp. $\mathcal{Y'}$) such that
\[ \varphi_* \mathcal{T} \simeq \varphi'_* \mathcal{T}' \]
and
\[ \End_{\mathcal{Y}}(\mathcal{T}) \simeq \End_{\mathcal{Y}'}(\mathcal{T}'). \]
Let $x \in X$ be the unique fixed closed point and let $\widehat{X} := \Spec \widehat{R}$,
where $\widehat{R} := \widehat{\stsh_{X,x}}$ is the completion of $\stsh_{X,x}$ with respect to the unique maximal ideal.
Then the canonical morphism $\widehat{X} \to X$ factors through the morphism $\mathcal{X} \to X$
and hence we have a diagram
\[ \begin{tikzcd}
Y \times_X \widehat{X} \arrow[d, "\widehat{\phi}"] \arrow[r] & \mathcal{Y} \arrow[d, "\varphi"] \arrow[r] & Y \arrow[d, "\phi"] \\
\widehat{X} \arrow[r, "\iota"] \arrow[rr, bend right=20, "\kappa"'] & \mathcal{X} \arrow[r] & X
\end{tikzcd} \]
such that all squares are cartesian.
Thus the following Lemma \ref{lem loc tilt} implies that
the scheme $\widehat{Y} := Y \times_X \widehat{X}$ admits a tilting bundle $\widehat{T}$ such that $\widehat{\phi}_* \widehat{T} \simeq \iota^* \varphi_* \mathcal{T}$.
Similarly, the scheme $\widehat{Y}' := Y' \times_X \widehat{X}$ admits a tilting bundle $\widehat{T}'$ such that $\widehat{\phi}'_* \widehat{T}' \simeq \iota^* \varphi'_* \mathcal{T}'$.
In particular, we have
\[ \widehat{\phi}_* \widehat{T} \simeq \widehat{\phi}'_* \widehat{T}' \]
and
\[ \End_{\widehat{Y}}(\widehat{T}) \simeq \End_{\widehat{Y}'}(\widehat{T}'). \]

Then Theorem \ref{Kal08 thm1.8} implies that there is a $\Gm$-equivariant tilting bundle $T$ (resp. $T'$) on $Y$ (resp. $Y'$) such that the pull-back of $T$ (resp. $T'$) on $\widehat{Y}$ (resp. $\widehat{Y}'$) is isomorphic to $\widehat{T}$ (resp. $\widehat{T}'$).
In addition, Kaledin constructed $\Gm$-actions on $\widehat{\phi}_* \widehat{T}$ and $\widehat{\phi}'_* \widehat{T}'$
such that there are $\Gm$-equivariant isomorphisms
\[ \widehat{\phi}_* \widehat{T} \simeq \kappa^*\phi_*T ~ \text{and} ~ \widehat{\phi}'_* \widehat{T}' \simeq \kappa^*\phi'_*T'. \]
To show that the $\Gm$-equivariant structures of $\widehat{\phi}_* \widehat{T}$ and $\widehat{\phi}'_* \widehat{T}'$ are same under the
isomorphism $\widehat{\phi}_* \widehat{T} \simeq \widehat{\phi}'_* \widehat{T}'$,
we have to recall Kaledin's construction of $\Gm$-equivariant structures.
In the following, we provide the detail of his argument, because there is only a sketch of the proof  in \cite{Kal08}.
Put $\widehat{M} := \widehat{\phi}_* \widehat{T}$.

\vspace{2mm}

\noindent
\textbf{Step 2. $\Gm$-equivariant structure of $\widehat{M}$}.

Let us consider the $\Gm$-action $\Gm \to \Aut(\widehat{R})$.
It is well-known that we can regard $\Aut(\widehat{R})$ as an open subscheme of the Hilbert scheme $\Hilb(\Spec(\widehat{R} \otimes \widehat{R})$)
via the map taking the graph of an automorphism of $\Spec \widehat{R}$.
Thus the tangent space of $\Aut(\widehat{R})$ at the identity $\id \in \Aut(\widehat{R})$ is isomorphic to
\[ \Ext^1_{\widehat{X} \times \widehat{X}}(\stsh_{\Delta}, \stsh_{\Delta}) \simeq \HH^1(\widehat{R}) \simeq \Der_{\CC}(\widehat{R}, \widehat{R}), \]
where $\Delta \subset \widehat{X} \times \widehat{X}$ is the diagonal and $\HH^1(\hat{R})$ is the 1st Hochschild cohomology of $\widehat{R}$.
Thus, the $\Gm$-action $\Gm \to \Aut(\widehat{R})$ determines a (non-zero) derivation
\[ \xi : \widehat{R} \to \widehat{R} \]
uniquely up to scalar multiplication.
Let $\CC[\varepsilon]$ be the ring of dual numbers and put
\[ \widehat{R}^{(1)} := \widehat{R} \otimes_{\CC} \CC[\varepsilon], ~ \widehat{Y}^{(1)} := \widehat{Y} \times \Spec \CC[\varepsilon], 
~ \text{and} ~ \widehat{Y}'^{(1)} := \widehat{Y}' \times \Spec \CC[\varepsilon]. \]
By deformation theory, the derivation $\xi : \widehat{R} \to \widehat{R}$ determines an automorphism
\[ f_{\xi} : \widehat{R}^{(1)} \xrightarrow{\sim} \widehat{R}^{(1)} \]
of $\widehat{R}^{(1)}$ such that $f_{\xi} \otimes_{\CC[\varepsilon]} \CC \simeq \id_{\widehat{R}}$ (see \cite[Lemma 1.2.6]{Ser}), and then, by taking pull-back, we have isomorphisms
of schemes
\[ F_{\xi} : \widehat{Y}^{(1)} \xrightarrow{\sim} \widehat{Y}^{(1)} ~ \text{and} ~ F'_{\xi} : \widehat{Y}'^{(1)} \xrightarrow{\sim} \widehat{Y}'^{(1)}. \]
Via an isomorphism
\[ \Der_{\CC[\varepsilon]}(\widehat{R}^{(1)}, \widehat{R}) \simeq \Der_{\CC}(\widehat{R}, \widehat{R}), \]
we have a derivation $\tilde{\xi} : \widehat{R}^{(1)} \to \widehat{R}$ and then the automorphism is given by
\[ f_{\xi}(\tilde{r}) = \tilde{r} + \tilde{\xi}(\tilde{r}) \otimes \varepsilon. \]

Let $\widehat{T}^{(1)}$ be the pull-back of $\widehat{T}$ by the projection $\widehat{Y}^{(1)} \to \widehat{Y}$,
and we also define a bundle $\widehat{T}'^{(1)}$ on $\widehat{Y}'^{(1)}$ in the same way.
Then $\widehat{T}^{(1)}$ and ${F_{\xi}}_*\widehat{T}^{(1)}$ are two lifts of $\widehat{T}$ to $\widehat{Y}^{(1)}$.
Since $\widehat{T}$ is tilting, we have $\Ext_{\widehat{Y}}^1(\widehat{T}, \widehat{T}) = 0$ and hence there is an isomorphism
\[ g : \widehat{T}^{(1)} \xrightarrow{\sim} {F_{\xi}}_*\widehat{T}^{(1)} \]
(see \cite[Theorem 7.1 (c)]{Har2}).
If we put $\widehat{M}^{(1)} := H^0(\widehat{Y}^{(1)}, \widehat{T}^{(1)})$, then we have a $\CC$-linear bijective morphism
\[ g_{\xi} : \widehat{M}^{(1)} :=  H^0(\widehat{Y}^{(1)}, \widehat{T}^{(1)}) \xrightarrow[g]{\sim} H^0(\widehat{Y}^{(1)}, {F_{\xi}}_*\widehat{T}^{(1)}) \xrightarrow[\text{adj}]{\sim} H^0(\widehat{Y}^{(1)}, \widehat{T}^{(1)}) = \widehat{M}^{(1)} \]
of $\widehat{M}^{(1)}$.
Note that, by construction, we have $g_{\xi} \otimes_{\CC[\varepsilon]} \CC = \id_{\widehat{M}}$ and 
\[ g_{\xi}(\tilde{r}\tilde{m}) = f_{\xi}(\tilde{r})g_{\xi}(\tilde{m}) \]
for $\tilde{r} \in \widehat{R}^{(1)}$ and $\tilde{m} \in \widehat{M}^{(1)}$.
From the first equality, for any $\tilde{m} \in \widehat{M}^{(1)}$, there is an element $\tilde{\xi}_{\widehat{M}}(\tilde{m}) \in \widehat{M}$ such that
\[ g_{\xi}(\tilde{m}) = \tilde{m} + \tilde{\xi}_{\widehat{M}}(\tilde{m}) \otimes \varepsilon. \]
This correspondence $\tilde{m} \mapsto \tilde{\xi}_{\widehat{M}}(\tilde{m})$
defines a $\CC$-linear map
\[ \tilde{\xi}_{\widehat{M}} : \widehat{M}^{(1)} \to \widehat{M}. \]
In addition, since we have
\begin{align*}
\tilde{r}\tilde{m} + \tilde{\xi}_{\widehat{M}}(\tilde{r}\tilde{m}) \otimes \varepsilon &= g_{\xi}(\tilde{r}\tilde{m}) \\
&= f_{\xi}(\tilde{r})g_{\xi}(\tilde{m}) \\
&= (\tilde{r} + \tilde{\xi}(\tilde{r}) \otimes \varepsilon)(\tilde{m} + \tilde{\xi}_{\widehat{M}}(\tilde{m}) \otimes \varepsilon) \\
&= \tilde{r}\tilde{m} + (\tilde{r} \tilde{\xi}_{\widehat{M}}(\tilde{m}) + \tilde{\xi}(\tilde{r}) \tilde{m}) \otimes \varepsilon,
\end{align*}
we have an equality
\[ \tilde{\xi}_{\widehat{M}}(\tilde{r}\tilde{m}) = \tilde{r} \tilde{\xi}_{\widehat{M}}(\tilde{m}) + \tilde{\xi}(\tilde{r}) \tilde{m}. \]
Thus, restricting $\tilde{\xi}_{\widehat{M}}$ to the subgroup $\widehat{M} \simeq \widehat{M} \otimes_{\CC} \varepsilon \subset \widehat{M}^{(1)}$, 
we have a $\CC$-linear endomorphism
\[ \xi_{\widehat{M}} \in \End_{\CC}(\widehat{M}) \]
such that
\[ \xi_{\widehat{M}}(rm) = \xi(r)m + r \xi_{\widehat{M}}(m) \]
for $r \in \widehat{R}$ and $m \in \widehat{M}$.

Let $\widehat{R} \oplus \widehat{M}$ be the square-zero extension, whose multiplication is given by
\[ (r,m)(r',m') := (rr', rm' + r'm). \]
Then one can check that the map
\[ \xi : \widehat{R} \oplus \widehat{M} \ni (r, m) \mapsto (\xi(r), \xi_{\widehat{M}}(m)) \in \widehat{R} \oplus \widehat{M} \]
defines a derivation of $\widehat{R} \oplus \widehat{M}$.
Furthermore, by taking restriction, $\xi$ also defines a derivation of the algebra
\[ (\widehat{R} / \mathfrak{m}^n \widehat{R}) \oplus (\widehat{M} / \mathfrak{m}^n\widehat{M}) \]
for $n > 0$.
If $n = 1$, then the module $\widehat{M} / \mathfrak{m}\widehat{M}$ is finite dimensional, and thus there is a representation of $\Gm$
\[ \pi : \Gm \to \GL(\widehat{M} / \mathfrak{m}\widehat{M}) \]
whose differential is the restriction of $\xi_{\widehat{M}}$.
Using this representation, we have an action
\[ \Gm \to \Aut_{\CC}((\widehat{R} / \mathfrak{m} \widehat{R}) \oplus (\widehat{M} / \mathfrak{m}\widehat{M})) \]
of $\Gm$ on the algebra $(\widehat{R} / \mathfrak{m} \widehat{R}) \oplus (\widehat{M} / \mathfrak{m}\widehat{M})$ such that
\[ c \cdot (\bar{r}, \bar{m}) := (\bar{r}, \pi(c)\bar{m}) \]
for $c \in \Gm$ and $(\bar{r}, \bar{m}) \in (\widehat{R} / \mathfrak{m} \widehat{R}) \oplus (\widehat{M} / \mathfrak{m}\widehat{M})$.

Then by using \cite[Lemma 5.2]{Kal08} inductively on $n$, we have an action of $\Gm$ on $ \widehat{R} \oplus \widehat{M}$ and
therefore we have a $\Gm$-equivariant structure of $\widehat{M}$.

\vspace{3mm}

\noindent
\textbf{Step 3. Comparing $\Gm$-equivariant structures}.

From the construction above, we notice that the $\Gm$-equivariant structure of $\widehat{M}$ depends on the choice of the isomorphism
\[ g : \widehat{T}^{(1)} \xrightarrow{\sim} {F_{\xi}}_*\widehat{T}^{(1)} \]
such that $g \otimes_{\CC[\varepsilon]} \CC = \id$, and the ambiguity of the choice of such isomorphisms is resolved by the group $\End_{\widehat{Y}}(\widehat{T})$ (see \cite[Theorem 7.1 (a)]{Har2}).
Since we have
\[ \End_{\widehat{Y}}(\widehat{T}) \simeq \End_{\widehat{R}}(\widehat{M}) \simeq \End_{\widehat{Y}'}(\widehat{T}'), \]
we can choose isomorphisms
\[ g : \widehat{T}^{(1)} \xrightarrow{\sim} {F_{\xi}}_*\widehat{T}^{(1)} ~ \text{and} ~ g' : \widehat{T}'^{(1)} \xrightarrow{\sim} {F'_{\xi}}_*\widehat{T}'^{(1)} \]
such that they give the same $\CC$-linear map
\[ g_{\xi} : \widehat{M}^{(1)} \to \widehat{M}^{(1)} \]
after taking the global sections.
This means that we can take $\Gm$-equivariant tilting bundles $T$ and $T'$ such that there are $\Gm$-equivariant isomorphisms
\[ \widehat{\phi}_* \widehat{T} \simeq \kappa^*\phi_*T \simeq \kappa^*\phi'_*T' \simeq \widehat{\phi}'_* \widehat{T}', \]
as desired.

By Lemma \ref{Kal08 lem5.3}, the $\Gm$-equivariant isomorphism  $\widehat{\phi}_* \widehat{T} \simeq \widehat{\phi}'_* \widehat{T}'$ implies that we have
\[ \phi_* T \simeq \phi'_*T \]
and the isomorphism $\End_{\widehat{Y}}(\widehat{T}) \simeq \End_{\widehat{Y}'}(\widehat{T}')$ implies that we have
\[ \End_Y(T) \simeq \End_{Y'}(T'). \]
Therefore we have the result.
\end{proof}

\begin{rem} \rm
We also have an equivalence between $\Gm$-equivariant derived categories of $Y$ and $Y'$ (see, for example, \cite[Section 4.2]{Kar15}).
\end{rem}

\begin{lem} \label{lem loc tilt}
Let $\psi : Z \to \Spec S$ be a projective morphism, $\mathfrak{p} \in  \Spec S$ a point, $\mathcal{Z} := Z \times_{\Spec S} \Spec \widehat{S_{\mathfrak{p}}}$,
and $\iota : \mathcal{Z} \to Z$ the canonical morphism.
If $T$ is a tilting bundle on $Z$, then $\iota^*T$ is a tilting bundle on $\mathcal{Z}$.
\end{lem}

\begin{proof}
Since $\Spec \widehat{S_{\mathfrak{p}}} \to \Spec S$ is flat and $\Ext^i_Z(T,T) \simeq R^i\psi_*(T^{\vee} \otimes T)$, 
it follows from the flat base change formula that $\iota^*T$ is a partial tilting bundle.

Next we show that $\iota^*T$ is a generator of $\mathrm{D}^-(\Qcoh(\mathcal{Z}))$.
Let $F \in \mathrm{D}^-(\Qcoh(\mathcal{Z}))$ be a complex such that $\RHom_{\mathcal{Z}}(\iota^*T, F) = 0$.
Since $\iota$ is quasi-compact and quasi-separated, we can define a functor $\iota_* : \Qcoh(\mathcal{Z}) \to \Qcoh(Z)$.
In addition, since the morphism $\iota$ is affine, the functor $\iota_*$ is exact and hence we can consider a functor between derived categories
\[ \iota_* : \mathrm{D}^-(\Qcoh(\mathcal{Z})) \to \mathrm{D}^-(\Qcoh(Z)), \]
which is the right adjoint of $\iota^*$.
Thus we have an isomorphism
\[ \RHom_{Z}(T, \iota_*F) \simeq \RHom_{\mathcal{Z}}(\iota^*T, F) = 0, \]
and this shows that $\iota_*F = 0$.
Since $\iota$ is affine, $\iota_*F = 0$ implies $F = 0$.
\end{proof}

As in Corollary \ref{cor thm form defo}, we can relax the assumption in Theorem \ref{main thm equiv} if the codimension of the exceptional locus is greater than or equal to three.

\begin{cor}
Under the same set-up of Theorem \ref{main thm equiv}, assume in addition that
\[ \codim_Y(\exc(\phi_0)) \geq 3. \]
Then any strict tilting-type equivalence for $Y_0$ and $Y'_0$ lifts to a tilting type equivalence for $Y$ and $Y'$.
\end{cor}

Let $p : Y \to \Spec D$, $p' : Y' \to \Spec D$ and $q : X \to \Spec D$ be the morphisms associated to the deformation.

\begin{cor} \label{cor equiv defo fiber}
Under the same assumption as in Theorem \ref{main thm equiv}, assume that there exists a good and strict tilting-type equivalence between $\D(Y_0)$ and $\D(Y'_0)$.
Then there is a good tilting-type equivalence between $\D(p^{-1}(t))$ and $\D(p'^{-1}(t))$ for any closed point $t \in \Spec D$.
\end{cor}

\begin{proof}
Let $T_0$ and $T'_0$ are tilting bundle that determines the good and strict tilting-type equivalence between $\D(Y_0)$ and $\D(Y'_0)$.
Then by Theorem \ref{main thm equiv} we have a tilting bundle $T$ (resp. $T'$) on $Y$ (resp. $Y'$) such that
\[ \End_Y(T) \simeq \End_{Y'}(T'). \]
Let us consider the following diagram.
\[ \begin{tikzcd}
p^{-1}(t) \arrow[d, hook] \arrow[r, "\phi"] & q^{-1}(t) \arrow[d, hook] \arrow[r] & \Spec \CC \arrow[d, "t"] \\
Y \arrow[r, "\phi"] \arrow[rr, "p"', bend right=20]& X \arrow[r, "q"] & \Spec D.
\end{tikzcd} \]
Since we have an isomorphism
\begin{align*}
(R\phi_*F)|_{q^{-1}(t)} \simeq \RGamma(p^{-1}(t), F|_{p^{-1}(t)})
\end{align*}
for an vector bundle $F$ on $Y$ applying the (derived) flat base change formula, we have that $T_t := T|_{p^{-1}(t)}$ is a good partial tilting bundle on $p^{-1}(t)$.
In addition it is clear that $T_t$ is a generator and hence $T_t$ is a good tilting bundle on $p^{-1}(t)$.

Similarly, a bundle $T'_t := T'|_{p'^{-1}(t)}$ is a tilting bundle on $p'^{-1}(t)$ such that
\[ \End_{p^{-1}(t)}(T_t) \simeq \End_{p'^{-1}(t)}(T'_t). \]
This shows the result.
Note that we also have $\phi_*(T_t) \simeq \phi'_*(T'_t)$.
\end{proof}

\section{Examples} \label{sect example}

The aim of this section is to provide some applications and counter-examples of the theorem we established in the sections above.

\subsection{Stratified Mukai flops and stratified Atiyah flops} \label{subsect; stratified flop}

Let $V = \CC^N$ be a $N$-dimensional vector space, $r$ an integer such that $1 \leq r \leq N-1$, and $\Gr(r, N)$ the Grassmannian of $r$-dimensional linear subspaces of $V$.
For each $r$ such that $2r \leq N$, we consider the following three varieties
\begin{align*}
Y_0 &:= \{ (L, A) \in \Gr(r, N) \times \End(V) \mid A(V) \subset L, A(L) = 0 \}, \\
Y_0' &:= \{ (L', A') \in \Gr(N - r, N) \times \End(V) \mid A'(V) \subset L', A'(L') = 0 \}, \\
X_0 &:= \overline{B(r)} := \{ A \in \End(V) \mid  A^2 =0, \dim \Ker A \geq N - r \}.
\end{align*}
The variety $Y_0$ has two projections $\phi_0 : Y_0 \to X$ and $\pi_0 : Y_0 \to \Gr(r, N)$.
Via the second projection $\pi_0$, we can identify $Y_0$ with the total space of the cotangent bundle on the Grassmannian $\Gr(r, N)$.

Similarly, $Y'_0$ has two projections $\phi'_0 : Y_0 \to X$ and $\pi'_0 : Y_0 \to \Gr(N - r, N)$, 
and the second projection arrows us to identify $Y'_0$ with the total space of the cotangent bundle on the Grassmannian $\Gr(N-r, N)$.

The affine variety $X_0$ is called a nilpotent orbit closure of type A.
The singularity of $X_0$ is known to be symplectic, and hence to be normal, Gorenstein, and canonical.

It is easy to see that two morphisms $\phi_0 : Y_0 \to X_0$ and $\phi_0' : Y'_0 \to X_0$ give resolutions of $X_0$.
Since $Y_0$ and $Y'_0$ are algebraic symplectic varieties, these two resolutions are crepant resolutions.

If $N \geq 3$ and $2r < N$, the diagram
\[ \begin{tikzcd}
Y_0 \arrow[dr, "\phi_0"] & & Y'_0 \arrow[ld, "\phi'_0"'] \\
& X_0 &
\end{tikzcd} \]
is a flop and this flop is called a \textit{stratified Mukai flop} on $\Gr(r,N)$.

Note that $Y_0$, $Y_0'$, and $X_0$ have natural $\Gm$-actions and $\phi_0$ and $\phi'_0$ are $\Gm$-equivariant.

These three varieties $Y_0$, $Y'_0$ and $X_0$ have natural one-parameter $\Gm$-equivariant deformations as follows.
\begin{align*}
Y &:= \{ (L, A, t) \in \Gr(r,N) \times \End(V) \times \CC \mid (A-t \cdot \id)(V) \subset L, (A + t \cdot \id)(L) = 0 \}, \\ 
Y' &:= \{ (L', A', t') \in \Gr(N - r,N) \times \End(V) \times \CC \mid (A' - t' \cdot \id)(V) \subset L', (A' + t' \cdot \id)(L') = 0 \}, \\
X &:=  \{ (A, t) \in \End(V) \times \CC \mid A^2 = t^2 \cdot \id, \dim \Ker (A - t \cdot \id) \geq N - r \}.
\end{align*}
Note that the variety $Y$ is isomorphic to the total space of a bundle $\widetilde{\Omega}_{\Gr(r,N)}$ on $\Gr(r, N)$ that lies on an exact sequence
\[ 0 \to \Omega_{\Gr(r, N)} \to \widetilde{\Omega}_{\Gr(r,N)} \to \stsh_{\Gr(r,N)} \to 0 \]
that gives a generator of
\[ H^1(\Gr(r, N), \Omega_{\Gr(r,N)}) = \CC. \]
The corresponding statement holds for $Y'$.
Put
\[ \phi : Y \ni (L,A,t) \mapsto (A, t) \in X ~ \text{and} ~  \phi' : Y' \ni (L',A',t') \mapsto (A', -t') \in X. \]
Then the morphisms $\phi : Y \to X$ and $\phi' : Y' \to X$ give two crepant resolutions of $X$, and the diagram
\[ \begin{tikzcd}
Y \arrow[dr, "\phi"] & & Y' \arrow[ld, "\phi'"'] \\
& X &
\end{tikzcd} \]
is a flop called a \textit{stratified Atiyah flop} on $\Gr(r,N)$.

In \cite{CKL10, CKL13}, Cautis, Kamnitzer, and Licata proved that there are derived equivalences for stratified Mukai flops $\D(Y_0) \simeq \D(Y'_0)$
(later we refer to this equivalence as CKL's equivalence).
Their equivalence was given as a Fourier-Mukai transform and was obtained as a corollary of their framework of \textit{categorical $\mathfrak{sl}_2$ action}.
In \cite{Cau12}, Cautis studied an explicit description of the Fourier-Mukai kernel, and as an application of it,
he show that CKL's equivalence for a stratified Mukai flop extends to an equivalence for a stratified Atiyah flop.

In the following, we add some results for derived equivalences for stratified Mukai flops and stratified Atiyah flops from the point of view of tilting bundles.

\begin{thm} \label{tilting SMF}
Assume that $2r < N$.
Then there exists a strict tilting-type equivalence for a stratified Mukai flop on $\Gr(r,N)$.
\end{thm}

\begin{proof}
According to Theorem \ref{Kal equiv}, it is enough to show that there is a good $\Gm$-action on $X_0 = \overline{B(r)}$.

Let $e_1, \dots, e_N$ be a standard basis of $V = \CC^N$ and $f_1, \dots, f_N$ the dual basis of $V^{\vee}$.
We regard $x_{ij} := e_i \otimes f_j \in V \otimes_{\CC} V^{\vee}$ as a variable of the affine coordinate ring of $\End(V) = V^{\vee} \otimes V$.
Then the affine coordinate ring $R_0$ of $X_0$ is a quotient of a polynomial ring $\CC[(x_{ij})_{i,j}]$ and the maximal ideal $\mathfrak{m}_o$ corresponding to 
the origin $o \in X_0$ is the ideal generated by the image of $\{ x_{ij} \}_{i,j}$.

Let us consider an action of $\Gm$ on $X_0$ given by $t \cdot A := t^{-1}A$ for $t \in \Gm$ and $A \in X_0$.
Clearly this action has a unique fixed point $o \in X$.
As a $\Gm$-representation, $\mathfrak{m}_o$ splits into the direct sum of lines spanned by a monomial.
Let $r \in R_0$ be a monomial of degree $d \geq 1$.
Then, for $t \in \Gm$ and $A \in X_0$, we have
\[ r^t(A) = r(t^{-1} \cdot A) = r(tA) = t^dr(A). \]
Thus the action of $\Gm$ on $\mathfrak{m}_o$ is positive weight and hence the action of $\Gm$ on $X_0$ is good.
\end{proof}

As an application of Theorem \ref{main thm equiv}, we have the following result.

\begin{thm} \label{Mukai Atiyah lift}
Assume that $2r \leq N - 1$.
Then, any good and strict tilting-type equivalence between $\D(Y_0)$ and $\D(Y_0')$ lifts to a good tilting-type equivalence between $\D(Y)$ and $\D(Y')$.

In addition, if $2r \leq N -2$,
then any strict tilting-type equivalence between $\D(Y_0)$ and $\D(Y_0')$ lifts to a good tilting-type equivalence between $\D(Y)$ and $\D(Y')$.
\end{thm}

\begin{proof}
As in Theorem \ref{tilting SMF}, we can check that there exists a good $\Gm$-action on $X$ which lifts to $\Gm$-actions on $Y$ and $Y'$.
Note that $\dim X_0 = \dim \overline{B(r)} = 2r(N-r)$ and $\Sing(X_0) = \overline{B(r-1)}$.
Thus we have
\begin{align*}
\codim_{X_0} \Sing(X_0) = 2r(N-r) - 2(r-1)(N - r + 1) = 2(N-2r+1).
\end{align*}
Therefore $\codim_{X_0} \Sing(X_0) \geq 3$ if and only if $2r < N$.

Furthermore, if we put
\[ B(k) := \{ A \in \End(V) \mid  A^2 =0, \dim \Ker A = N - k \}, \]
the fiber $\phi_0^{-1}(A)$ of $A \in B(k)$ ($k<r$) is isomorphic to $\Gr(r - k, \Ker(A)/ \mathrm{Im}(A))$ and thus we have
\[ \dim \phi_0^{-1}(A) = (r - k)(N - r - k). \]
Therefore we have
\[ \dim \phi_0^{-1}(B(k)) = 2k(N-k) + (r - k)(N-r-k) \]
for $0 \leq k \leq r -1$ and hence we have
\begin{align*}
 \codim_{Y_0} \overline{\phi_0^{-1}(B(k))} &= 2r(N-r) - 2k(N-k) - (r - k)(N-r-k)\\
&=\left(k- \frac{N}{2}\right)^2 - \frac{N^2}{4} + rN - r^2 \\
&\geq N - 2r + 1.
\end{align*}
and the equality holds if $k = r - 1$.
This shows that $\codim_{Y_0}(\exc(\phi_0)) \geq 3$ if and only if $2r \leq N - 2$.

Let $T$ and $T'$ be tilting bundles that give the tilting-type equivalence between $\D(Y)$ and $\D(Y')$.
Then by Theorem \ref{main thm equiv} we have $\phi_*T \simeq \phi'_*T'$.
This means that the tilting-type equivalence is strict.
\end{proof}

In conclusion, we have the following result.

\begin{thm}
If $2r \leq N - 2$, then there exists a tilting-type equivalence for the stratified Atiyah flop on $\Gr(r, N)$.
\end{thm}

The author expects that the same statement holds for remaining cases:

\begin{conj}
Assume that $2r = N -1$ or $2r = N$.
There exists a tilting-type equivalence for the stratified Atiyah flop on $\Gr(r, N)$.
\end{conj}

This conjecture is true in the following low dimensional cases:

\begin{eg} \rm
If $N = 2$ and $r = 1$, the stratified Atiyah flop is the $3$-fold Atiyah flop.
If $N = 3$ and $r = 1$, the stratified Atiyah flop on $\Gr(1,3)$ is the usual standard flop of $4$-folds.
In these cases, the conjecture above is known to be true.
\end{eg}

The author also expects that CKL's equivalence for a stratified Mukai flop is tilting-type.
Indeed, if $r=1$, CKL's equivalence for a Mukai flop is tilting-type \cite{Cau12, H17a}.
In addition, as noted above, Cautis proved that CKL's equivalence extends to an equivalence for a stratified Atiyah flop \cite[Theorem 4.1]{Cau12}
as in our Theorem \ref{Mukai Atiyah lift}.

Finally we note that the discussions above also show the following result.

\begin{thm}
The nilpotent orbit closure $X_0 := \overline{B(r)}$ (resp. its $\Gm$-equivariant deformation $X$) admits an NCCR for all $2r \leq N$ that is derived equivalent to $Y_0$ and $Y'_0$ (resp. $Y$ and $Y'$).
\end{thm}

\begin{proof}
We can take a lift $T$ of a tilting bundle $T_0$ on $Y_0$ to $Y$ without assuming that $T_0$ is good or $\codim_{X_0} \Sing(X_0) \geq 3$.
Then the result follows from Lemma \ref{from tilting to NCCR}.
As noted above, derived equivalences for $Y$ and $Y'$ was proved by Cautis.
\end{proof}

\subsection{A counter-example} \label{sect Eg counter}
In the present subsection, we provide an example that suggests we cannot remove the assumption $\codim_{X_0} \Sing(X_0) \geq 3$ in Theorem \ref{main thm equiv}.

Let us consider the case if $N = 2$ and $r=1$ in the subsection above.
Then $X_0 = \overline{B(1)}$ has a du Val singularity of type $A_1$, and $Y_0 = Y_0'$ is the total space of a line bundle $\stsh_{\PP^1}(-2)$ on $\PP^1$.
Moreover the singularity of $X$ is a threefold $A_1$-singularity, and $Y$ and $Y'$ are isomorphic to the total space of $\stsh_{\PP^1}(-1)^{\oplus 2}$ as abstract varieties.
However there is no isomorphism $f : Y \xrightarrow{\sim} Y'$ such that $\phi = \phi' \circ f$.
A bundle $T_0 = \stsh_{Y_0} \oplus \stsh_{Y_0}(-1)$, where $\stsh_{Y_0}(-1)$ is a pull-back of $\stsh_{\PP^1}(-1)$ to $Y_0$, is a good tilting bundle on $Y_0$.

A pair of tilting bundles $(T_0, T_0)$ provides a good and strict tilting-type equivalence $\D(Y_0) \simeq \D(Y_0)$, which is identity.
$T_0$ lifts to a good tilting bundle $T = \stsh_Y \oplus \stsh_Y(-1)$  on $Y$, where $\stsh_Y(-1)$ is a pull-back of $\stsh_{\PP^1}(-1)$ to $Y$.
Similarly, $T_0$ lifts to a tilting bundle $T' = \stsh_{Y'} \oplus \stsh_{Y'}(-1)$  on $Y'$, where $\stsh_{Y'}(-1)$ is a pull-back of $\stsh_{\PP^1}(-1)$ on $Y'$.

However we have $T|_{Y^o} \neq T'|_{Y^o}$, where $Y^o$ is the common open subset of $Y$ and $Y'$.
Indeed, we have $\phi_* \stsh_Y(-1) \neq \phi'_* \stsh_{Y'}(-1)$.

On the other hand, a pair of tilting bundles $(T_0, T_0^{\vee})$ induces a good and strict tilting-type equivalence $\D(Y_0) \xrightarrow{\sim} \D(Y_0)$, which is a spherical twist around a sheaf $\stsh_{\PP^1}(-1)$ on the zero-section $\PP^1 \subset Y_0$.

Since $T_0^{\vee}$ lifts to a bundle $T'^{\vee}$ on $Y'$ and one has $T|_{Y^o} \simeq T'^{\vee}|_{Y^o}$, the above equivalence lifts to a good and strict tilting-type equivalence
\[ \RHom_Y(T,-) \otimes^{\mathrm{L}}_{\End_Y(T)} T'^{\vee} : \D(Y) \xrightarrow{\sim} \D(Y'). \]

\appendix
\section{Symplectic resolutions} \label{sect appendix}

\subsection{Definitions and Properties}

In this subsection, we recall some basics of symplectic singularities.

\begin{defi}[\cite{Be00}] \rm
Let $X$ be an algebraic variety. We say that $X$  is a \textit{symplectic variety} if
\begin{enumerate}
\item[(i)] $X$ is normal.
\item[(ii)] The smooth part $X_{\mathrm{sm}}$ of $X$ admits a symplectic $2$-form $\omega$.
\item[(iii)] For every resolution $f : Y \to X$, the pull back of $\omega$ to $f^{-1}(X_{\mathrm{sm}})$ extends to a global holomorphic $2$-form on $Y$.
\end{enumerate}
Let $X$ be an algebraic variety. We say that a point $x \in X$ is a \textit{symplectic singularity} if there is an open neighborhood $U$ of $x$ such that $U$ is a symplectic variety.  
\end{defi}

Symplectic singularities belong to a good class of singularities that appears in minimal model theory.

\begin{prop}[\cite{Be00}]
A symplectic singularity is Gorenstein canonical.
\end{prop}

For symplectic singularities, we can consider the following reasonable class of resolutions.

\begin{defi} \rm
Let $X$ be a symplectic variety.
A resolution $\phi : Y \to X$ of $X$ is called \textit{symplectic} if the extended $2$-form $\omega$ on $Y$ is non-degenerate.
In other words, the $2$-form $\omega$ defines a symplectic structure on $Y$.
\end{defi}

It is easy to observe the following criteria for symplectic resolutions.

\begin{prop}
Let $X$ be a symplectic variety and $\phi : Y \to X$ a resolution. Then, the following statements are equivalent
\begin{enumerate}
\item[(1)] $\phi$ is a crepant resolution,
\item[(2)] $\phi$ is a symplectic resolution,
\item[(3)] the canonical divisor $K_Y$ of $Y$ is trivial.
\end{enumerate}
\end{prop}

\subsection{Derived equivalence}
In this subsection, we discuss the derived equivalence for symplectic resolutions.
First we recall the following theorem proved by Kaledin's \cite{Kal08}.
Note that the following theorem is not stated explicitly in \cite{Kal08},
but we can obtain it by analyzing the proof of Theorem 1.6 of loc. cit.

\begin{thm}[\cite{Kal08}, Theorem 1.6] \label{Kal08 thm1.6}
Let $X = \Spec R$ be an affine symplectic variety,
and let $\phi : Y \to X$ and $\phi' : Y' \to X$ be two symplectic resolutions of $X$.
Then, for every maximal ideal $\mathfrak{m} \subset R$,
there exists a strict tilting-type equivalence between $\D(Y \otimes \widehat{R_{\mathfrak{m}}})$ and $\D(Y' \otimes \widehat{R_{\mathfrak{m}}})$
defined by tilting bundles $(\EE, \EE')$, satisfying
\[ (\phi \otimes \widehat{R_{\mathfrak{m}}})_*\EE \simeq (\phi' \otimes \widehat{R_{\mathfrak{m}}})_* \EE'. \]
\end{thm}

Since the construction of tilting bundles is very complicated, it is not clear whether the tilting bundle he constructed is good or not (at least for the author).

\begin{thm} \label{Kal equiv}
Let $X = \Spec R$ be an affine symplectic variety, and $\phi : Y \to X$ and $\phi : Y' \to X$ two symplectic resolutions of $X$.
Assume that $X$ admits a good $\Gm$-action.
Then there exists a strict tilting-type equivalence between $\D(Y)$ and $\D(Y')$.
\end{thm}

\begin{proof}
First we note that the $\Gm$-action on $X$ lifts to a $\Gm$-action on $Y$ and $Y'$ \cite[Theorem 1.8 (i)]{Kal08}.
Let $\mathfrak{m} \subset R$ be the maximal ideal that corresponds to a unique fixed point of $X$.
Let $\widehat{R}$ be the completion of $R$ with respect to $\mathfrak{m} \subset R$.
Put $\widehat{Y} := Y \times_X \Spec \widehat{R}$ and $\widehat{Y}' := Y' \times_X \Spec \widehat{R}$,
and let $\widehat{\phi} : \widehat{Y} \to \Spec \widehat{R}$ and $\widehat{\phi'} : \widehat{Y}' \to \Spec \widehat{R}$ be the projections.

Then, by Theorem \ref{Kal08 thm1.6}, there exist tilting bundles $\widehat{\EE}$ and $\widehat{\EE'}$ on $\widehat{Y}$ and $\widehat{Y}'$, respectively, 
such that
$\widehat{\phi}_* \widehat{\EE} \simeq \widehat{\phi'}_* \widehat{\EE'}$ as $\widehat{R}$-modules
and $\End_{\widehat{Y}}(\widehat{\EE}) \simeq \End_{\widehat{Y}'}(\widehat{\EE'})$ as $\widehat{R}$-algebras.

By Theorem \ref{Kal08 thm1.8}, there exist tilting bundles $\EE$ and $\EE'$ on $Y$ and $Y'$, respectively,
such that $\EE \otimes_R \widehat{R} \simeq \widehat{\EE}$ and $\EE' \otimes_R \widehat{R} \simeq \widehat{\EE'}$.
Since $\phi_* \EE \otimes_R \widehat{R} \simeq \widehat{\phi}_* \widehat{\EE}$ and $\phi'_* \EE' \otimes_R \widehat{R} \simeq \widehat{\phi'}_* \widehat{\EE'}$,
Lemma \ref{Kal08 lem5.3} and the similar argument as in Step 3 of the proof of Theorem \ref{main thm equiv} imply that we have
\[ \phi_* \EE \simeq \phi'_* \EE'. \]
Thus we have a strict tilting-type equivalence between $\D(Y)$ and $\D(Y')$.
\end{proof}

\end{document}